\documentclass[a4paper]{amsart}
\usepackage{amssymb}
\usepackage{verbatim}
\usepackage{amscd}
\def\card{{{\operatorname{card}}}}

\numberwithin{equation}{section}
\theoremstyle{plain}
\newtheorem{theorem}[equation]{Theorem}

\newtheorem{lemma}[equation]{Lemma}
\newtheorem*{Lemma A}{Lemma A}
\newtheorem*{Lemma B}{Lemma B}
\newtheorem*{Lemma C}{Lemma C}
\theoremstyle{definition}

\theoremstyle{remark}

\begin{document}
\title [A construction of subshifts and a class of semigroups]{A construction of subshifts and \\ a class of semigroups}
\author{Toshihiro Hamachi}
\author{Wolfgang Krieger}
\begin{abstract}
Subshifts with property $(A)$ are constructed from a class of directed graphs. As special cases the Markov-Dyck shifts are shown to have property $(A)$. The 
semigroups, that are associated to $\mathcal R$-graph shifts with Property (A), are determined.
Also results on the reconstruction of $\mathcal R$-graphs from their $\mathcal R$-graph shifts are obtained.
\end{abstract}
\maketitle

\section{Introduction}
Let $\Sigma$ be a finite alphabet, and let $S$ be the shift 
on the shift space $\Sigma^{\Bbb Z}$,
$$
S((x_{i})_{i \in \Bbb Z}) =  (x_{i+1})_{i \in \Bbb Z}, 
\qquad 
(x_{i})_{i
\in \Bbb Z}  \in \Sigma^{\Bbb Z}.
$$
An 
$S$-invariant closed subset $X$ of $\Sigma^{\Bbb Z}$ is called a subshift. For an introduction to the theory 
of subshifts see \cite {Ki} or  \cite{LM}. 
In \cite {Kr2}
a Property $(A)$ of subshifts was introduced that is an invariant of topological conjugacy. Also in  \cite {Kr2} a semigroup was constructed that is invariantly attached to a subshift with property $(A)$ (see also \cite[Section 9]{CS}).
Prototypes of  subshifts with Property $(A)$ are the Dyck shifts \cite {Kr1}. To recall the construction of the Dyck shifts, let $N> 1$, and let $\alpha^-(n), \alpha^+(n), 0 \leq n < N,$ be the generators of the Dyck inverse monoid (the polycyclic monoid \cite {NP}) $\mathcal D_N$, that satisfy the relations
$$
\alpha^-(n) \alpha^+(m) =
\begin{cases}
1, &\text{if  $n = m$}, \\
0, &\text {if $n \neq m$}.
\end{cases}
$$
The Dyck shifts  are defined as the subshifts
$$
D_N \subset( \{ \alpha^-(n): 0 \leq n < N \} \cup\{ \alpha^+(n):0 \leq n < N \})^\Bbb Z
$$
with the admissible words $(\sigma_i)_{1 \leq i \leq I  } , I \in \Bbb N,$ of $D_N, N > 1,$ given by the condition 
\begin{align*}
\prod_{1 \leq i \leq I } \sigma_i \neq 0. 
\end{align*}
The Dyck inverse monoid $\mathcal D_N$ is associated to the Dyck shift $D_N$.

We denote a finite directed graph  with
vertex set ${\frak P}$ and edge set 
${\mathcal E}$ by $G(\frak P, \mathcal E)$.
As notation for the source vertex and target vertex of an edge or path in a directed graph we use $s$ and $t$.
We recall from \cite{Kr4} the notion of  
an $\mathcal R$-graph.
Let there be given 
a finite  directed graph $G(\mathcal V, \mathcal E)$. Assume also given a partition 
$$
\mathcal E = \mathcal E^-  \cup\mathcal E^+.
$$
We set
\begin{align*}
& \mathcal E^- (\frak q,\frak r) = \{ e^- \in  \mathcal E^- : s(e^-) = \frak q,\  t( e^-) =  \frak r \},
\\
& \mathcal E^+(\frak q,\frak r) = \{ e^- \in  \mathcal E^+ : s(e^+) = \frak r,\  t( e^+) =  \frak q \}, \qquad  \frak q,\frak r \in \frak P.
\end{align*}
 We assume that $ \mathcal E^- (\frak q,\frak r)  \neq \emptyset$ if and only if $  \mathcal E^+(\frak q,\frak r) \neq \emptyset,  \frak q,\frak r \in \frak P$, and we assume that 
the  directed graph $G(\frak P  ,  \mathcal E^-   )$ is strongly connected, or, equivalently, that 
the  directed graph $G(\frak P  ,  \mathcal E^+ )$ is strongly connected.
We call  $G( \frak P ,\mathcal E^- \cup \mathcal E^+ ) $ a partitioned directed graph.
Let there further be given relations \footnote{We consider complete heterogeneous relations.} 
$$
\mathcal R    (\frak q,\frak r) \subset   \mathcal E^- (\frak q,\frak r)   \times   \mathcal E^+(\frak q,\frak r) , \qquad \frak q,\frak r \in \frak P,
$$
and set
$$
\mathcal R = \bigcup_{ \frak q,\frak r \in \frak P} \mathcal R    (\frak q,\frak r) .
$$
The resulting structure, that we call an $\mathcal R$-graph, we denote by 
$G_\mathcal R(\frak P, \mathcal E^-\cup\mathcal E^+)$. 
We also recall 
the construction of a semigroup (with zero)  $\mathcal S(G_\mathcal R(\frak P,   \mathcal E^-\cup \mathcal E^+   ))$ 
from an $\mathcal R$-graph $G_\mathcal R(\frak P, \mathcal E^-\cup\mathcal E^+)$ as described in \cite {Kr3}.   
The semigroup $\mathcal S(G_\mathcal R(\frak P,   \mathcal E^-\cup \mathcal E^+))$ contains idempotents $\bold 1_{\frak p}, \frak p \in \frak P,$ and  has $\mathcal E$ as a generating set.
Besides $\bold 1_{\frak p}^2 = \bold 1_{\frak p},\frak p \in \frak P$, the defining relations are:
$$
f^- g^+ =\bold 1_{\frak q}, \quad f^- \in  \mathcal E^- (\frak q,\frak r), g^+ \in  \mathcal E^+(\frak q,\frak r) , (  f^- ,  g^+ ) \in \mathcal R    (\frak q,\frak r), \quad \frak q,\frak r \in\frak P,
$$
and
\begin{align*}
&\bold 1_{\frak q} e^- = e^- \bold 1_{\frak r} = e^-, \quad e^- \in  \mathcal E^- (\frak q,\frak r), \\
&\bold 1_{\frak r} e^+ = e^+ \bold 1_{\frak q} = e^+, \quad e^+ \in  \mathcal E^+ (\frak q,\frak r), \quad  \frak q,\frak r \in\frak P,
\end{align*}
$$
f^- g^+  = \begin{cases}{ \bold 1}_{{\frak q}}, &\text {if $ (  f^- ,  g^+ ) \in{ \mathcal R}({\frak q},{\frak r}) $,}\\
0, & \text{if $ (  f^- ,  g^+ )\notin {\mathcal R}({\frak q},{\frak r}), \quad f^- \in  {\mathcal E}^- ({\frak q},{\frak r}), g^+ \in { \mathcal E}^+({\frak q},{\frak r)},        
\
{\frak q},  {\frak r} \in{\frak P},
 $}
\end{cases}
$$
and
$$
\bold 1_{\frak q}\bold 1_{\frak r}= 0,    \quad  \frak q,\frak r \in \frak P,\frak q \neq\frak r .
$$
We call $\mathcal S_\mathcal R(G(\frak P,   \mathcal E^-\cup \mathcal E^+ ))$ an $ \mathcal R$-graph semigroup. 
We write $\mathcal S^-(\frak  P,   \mathcal E^- )$($\mathcal S^+(\frak P,   \mathcal E^+ )$) for the set of non-zero elements of the subsemigroup of 
$\mathcal S(G_\mathcal R(\frak  P,   \mathcal E^-\cup \mathcal E^+ ))$, that is generated by 
$ \mathcal E^-$
($\mathcal E^+$).

Special cases are the graph inverse semigroups of  finite directed graphs 
 $ G (\frak P,\mathcal E)$ (\cite {AH},\cite [Section 10.7.]{ L}). With the edge set 
 $\mathcal E^- = \{  e^-: e_\circ \in \mathcal E_\circ\}$ of a copy of 
 $G(\frak P,\mathcal E)$, and with the  edge set $\mathcal E^+ =
  \{  e^-: e \in \mathcal E\}$ of the reversal of $G(\frak P,\mathcal E)$, the 
 graph inverse semigroup 
 $\mathcal S( G(\frak P,\mathcal E))$ of $G(\frak P,\mathcal E)$
  is  the $\mathcal R$-graph semigroup of the partitioned graph 
 $G(\frak P, \mathcal E^-  \cup\mathcal  E^+ )  $ with the relations
$$
\mathcal R(\frak q, \frak r) = \{(e^- , e^+ ) : e  \in \mathcal E   , s(e ) = \frak q, t(e ) = \frak r \}, \quad   \frak q,\frak r \in \frak P.
$$
 
In \cite {HI} a criterion was given for the existence of an embedding of an irreducible subshift of finite type  into a Dyck shift and this result was extended in \cite {HIK} to a larger class of target shifts with Property $(A)$. These target shifts were constructed by a method that presents  the subshifts by means of a suitably structured irreducible finite labeled directed graph with labels taken from the inverse semigroup of an irreducible finite directed graph, in which every vertex has at least two incoming edges. This method was extended in \cite {Kr4} by the use of $\mathcal R$-graph semigroups. Following   \cite {HIK, Kr4}
we describe this construction.  

We denote a finite directed labelled graph  with
vertex set ${\mathcal V}$, edge set 
${\Sigma}$ and a label map $\lambda$ by $G(\mathcal V, \Sigma, \lambda)$. 
Let there be given an $\mathcal R$-graph 
$G_\mathcal R(\frak  P  ,  \mathcal E^-  \cup \mathcal E^+ )$ and 
a finite strongly connected labeled directed graph $G(\mathcal V, \Sigma, \lambda)$ such that 
\begin{align*}
 \lambda(\sigma) \in  S^-(\frak  P,   \mathcal E^- ) \cup \{ \bold 1_{\frak p}: \frak p \in \frak P\} \cup S^+(\frak P,   \mathcal E^+) ,\qquad \sigma \in \Sigma. \tag {G 1}
\end{align*}
The label map $\lambda$ extends to finite paths $(\sigma_i)_{1 \leq i \leq I}$ in the graph 
$G(\mathcal V, \Sigma)$ by
$$
\lambda((\sigma_i)_{1 \leq i \leq I}) = \prod_{1 \leq i \leq I}\lambda (\sigma_i).
$$
Denoting for $\frak p \in \frak P$ by $\mathcal V(\frak p) $ the set of $V\in \mathcal V$ such that there is a cycle $(\sigma_i)_{1 \leq i \leq I}, I \in\Bbb N,$ in the graph 
$G( \mathcal V, \Sigma)$ from $V$ to $V$ such that
$$ 
\lambda(( \sigma_i)_{1 \leq i \leq I}) = \bold 1_{\frak p}, 
$$
we require the following conditions (G 2 - 5) to be satisfied:

\begin{align*}
\mathcal V(\frak p) \neq \emptyset, \quad \frak p \in \frak P, \tag {G 2}
\end{align*}

\bigskip

\noindent(G 3)\quad\quad\quad\quad\quad\quad\quad\quad \ 
$ \{\mathcal V(\frak p) :\frak p \in \frak P\}$  is a partition of  $\mathcal V,
$

\bigskip

\noindent(G 4)
 For $ V  \in \mathcal V(\frak p), \frak p \in \frak P$,
 and for all edges $\sigma$ that leave $V, \bold 1_{\frak p} \lambda(\sigma) \neq 0$, 
and for all edges $\sigma$ that enter $V, \lambda(\sigma)\bold 1_{\frak p}  \neq 0$,

\bigskip

\noindent(G 5)
For $f\in \mathcal S(G_\mathcal R(\frak  P,   \mathcal E^-\cup \mathcal E^+  ) ),  \frak q,  \frak r \in 
\frak P$, such that $\bold 1_{\frak q}  f\bold 1_{\frak r}  \neq 0$,  and for 
$U\in \mathcal V (\frak q), W\in \mathcal V(\frak r)$, there exists
a path $b$ in the labeled directed graph  $G(\mathcal V,\Sigma, \lambda)$ from $U$ to $W$ such that 
 $\lambda(b) = f$.
 
\bigskip

A finite  labeled directed graph $G(\mathcal V,\Sigma, \lambda)$, that satisfies conditions (G  1 - 5), gives rise to a subshift 
$X(G(\mathcal V,\Sigma, \lambda))$,
that has as its language of admissible words the set of finite paths 
$b$
in  the graph $G(\mathcal V,\Sigma, \lambda)$   such that $\lambda (b) \neq 0$.
We  call this subshift  
an 
$\mathcal S(G_\mathcal R(\frak P, \mathcal E^-\cup  \mathcal E^+))$-presentation.
Given an $\mathcal R$-graph 
$G_\mathcal R(\frak P, \mathcal E^-\cup\mathcal E^+)$
and using the injection of the edge set $ \mathcal E^-\cup  \mathcal E^+$ into 
$\mathcal S(\frak P, \mathcal E^-\cup  \mathcal E^+) $
as label map, one obtains a  particular case of an 
$\mathcal S(G(\frak P, \mathcal E^-\cup  \mathcal E^+))$-presentation,  that we denote by
$X(G_\mathcal R(\frak P, \mathcal E^-\cup\mathcal E^+)),$
and that we call the $\mathcal R$-graph shift of 
$G_\mathcal R(\frak P, \mathcal E^-\cup\mathcal E^+)$.
In the case of the graph inverse semigroups  $\mathcal S( G(\frak P,\mathcal E))$ of strongly connected finite directed graphs 
$G(\frak P,\mathcal E)$  the subshifts   $X( G(\frak P,\mathcal E))$ are  the Markov-Dyck shifts \cite {M}.The Dyck shifts $D_N$ can be obtained in this way from the one-vertex directed graph with $N>1$ loops.
Also the Markov-Motzkin shifts \cite {KM1} of strongly connected finite directed graphs 
$\ G(\frak P,\mathcal E)$  can be written as  $\mathcal S( G(\frak P,\mathcal E))$-presentations.

B\'eal, Blockelet and Dima \cite{BBD1, BBD2} have introduced the notions of a Dyck automaton and of a sofic Dyck shift. Strengthening the Condition (G 1) to
\begin{align*}
\lambda(\sigma) \in   \mathcal E^-  \cup \{ \bold 1_{\frak p}: \frak p \in \frak P\} \cup    \mathcal E^+, \tag {1.1}
\end{align*}
one obtains directed labelled graphs $G(\mathcal V,\Sigma, \lambda)$, that are Dyck-automata, with $\mathcal S(G(\frak P, \mathcal E^-\cup  \mathcal E^+))$-presentations $X(\mathcal V,\Sigma, \lambda)$, that are sofic Dyck shifts. The alphabet of a sofic Dyck shift is partitioned into a set of call symbols, a set of internal symbols and a set of return symbols. The corresponding partition of the alphabet of an $\mathcal S$-presentation, that satisfies (1.1), is the partition of its alphabet into the sets
$ \mathcal E^- , \{ \bold 1_{\frak p}, \frak p \in \frak P\} , \mathcal E^+$.
The set of matched edges, that appears in the construction of a Dyck automaton, is provided by the relation $\mathcal R$. The $\mathcal R$-graph shifts are finite type Dyck shifts in the sense of \cite{BBD3}.

Given finite sets $\mathcal E^-$ and $\mathcal E^+$ and a relation 
$\mathcal R \subset \mathcal E^-  \times  \mathcal E^+  $, we set
$$
\mathcal E^-(\mathcal R) = \{e^-  \in  \mathcal E^-: \{e^-  \}\times \mathcal E^+\subset \mathcal R  \},
\quad\mathcal E^+(\mathcal R) = \{e^+  \in  \mathcal E^+: \mathcal E^-\times \{e^+  \}\subset \mathcal R  \}.
$$
For a partitioned 
directed
graph  $G(\frak P, \mathcal E^-\cup \mathcal E^+)$ denote by 
$\frak P^{(1)}$ the set of vertices in $\frak P$ that have a single predecessor vertex in $\mathcal E^-$, or, equivalently, that have a single successor vertex in $\mathcal E^+$. For $\frak p \in\frak P^{(1)}$ the predecessor vertex of  $\frak p $ in $\mathcal E^-$, which is identical to the successor vertex of  
$\frak p $ in  $\mathcal E^+$,
will be denoted by $\eta(\frak p)$.
For an $\mathcal R$-graph $G_\mathcal R(\frak P, \mathcal E^-\cup \mathcal E^+)$ 
we set
$$
\mathcal E^-_\mathcal R = \bigcup_{\frak p \in \frak P^{(1)}}\mathcal E^-(\mathcal R(\eta (\frak p), \frak p))  ,
\quad
\mathcal E^+_\mathcal R = \bigcup_{\frak p \in \frak P^{(1)}}\mathcal E^+(\mathcal R(\eta (\frak p), \frak p)) ,
$$
and
$$
{ \frak P}_\mathcal  R ^{(1)}=
\{ \frak p \in 
\frak P^{(1)}: 
 \mathcal R(\eta (\frak p), \frak p)    =   \mathcal E^-(\eta (\frak p), \frak p)  \times \mathcal E^+(\eta (\frak p), \frak p)   \}.
$$

An  $\mathcal S(G(\frak P, \mathcal E^-\cup  \mathcal E^+))$-presentation
$X(G(\mathcal V,\Sigma, \lambda))$ 
is a  Markov shift if and only if $ \frak P^{(1)}={\frak P}_\mathcal  R ^{(1)}$.
We formulate two conditions (I) and (II) on $\mathcal R$-graphs $G_\mathcal R(\frak P, \mathcal E^-\cup \mathcal E^+)$, such that
$$
\frak P^{(1)}\setminus { \frak P}_\mathcal  R ^{(1)}\neq \emptyset.
$$
Condition (II) comes in two parts (II$-$) and (II$+$) that  are symmetric to one another:

\bigskip

(I)
For  $\frak p \in  \frak P^{(1)}\setminus { \frak P}_\mathcal  R ^{(1)}$,
$
\mathcal E^-(\mathcal R(\eta (\frak p), \frak p))= \emptyset,
$
or
$
\mathcal E^+(\mathcal R(\eta (\frak p), \frak p)) 
= \emptyset.
$

\bigskip

(II$-$) There is no cycle in the directed graph $G(\frak P, \mathcal E^-_{\mathcal R}) $.

\bigskip

(II$+$)  
There is no cycle in the directed graph $G(\frak P, \mathcal E^+_{\mathcal R}) $.

\bigskip
 
We show in Section 2 that an 
$\mathcal S(G_\mathcal R(\frak P, \mathcal E^-\cup  \mathcal E^+))$-presentation has Property $(A)$ if and only if the $ \mathcal R$-graph 
$ G_\mathcal R(\frak P, \mathcal E^-\cup  \mathcal E^+)$ satisfies conditions (I)  and (II).
In particular the  $\mathcal R$-graph
shifts $X( G_\mathcal R(\frak P, \mathcal E^-\cup  \mathcal E^+))$
have Property $(A)$ if and only if the  $\mathcal R$-graph 
$G_\mathcal R(\frak P, \mathcal E^-,  \mathcal E^+)$ satisfies Conditions (I)  and (II).
This implies that Markov-Dyck shifts of strongly connected finite directed graphs
have Property $(A)$. Also the Markov-Motzkin shifts of strongly connected finite directed graphs 
have Property $(A)$. Concerning the invariance under flow-equivalence of Property $(A)$, and of the associated semigroup, in particular for $\mathcal R$-graph shifts, see \cite{Kr3}.
 
In Section 3 we describe how one can
obtain from an $\mathcal R$-graph 
$ G_\mathcal R(\frak P, \mathcal E^-\cup  \mathcal E^+)$, such that 
$\frak P^{(1)}\setminus { \frak P}_\mathcal  R ^{(1)}\neq \emptyset$, that satisfies conditions (I)  and (II), an $\mathcal R$-graph 
$ G_{\widehat {\mathcal R}}(\widehat{\frak P},\widehat{ \mathcal E}^-\cup \widehat{ \mathcal E}^+)$,  such that the $\mathcal R$-graph semigroup 
$\mathcal S(G_{\widehat {\mathcal R}}(\widehat{\frak P},\widehat{ \mathcal E}^-\cup \widehat{ \mathcal E}^+))$ is associated to all 
$\mathcal S(G_\mathcal R(\mathcal V, \mathcal E^-\cup  \mathcal E^+))$-presentation. To obtain the 
$\mathcal R$-graph 
$ G_{\widehat {\mathcal R}}({\frak P},\widehat{ \mathcal E}^-\cup \widehat{ \mathcal E}^+)$  we apply a procedure, that extends a procedure for Markov-Dyck shifts,  that was described in \cite{HK2} and \cite{KM2}.

In Section 5 we consider examples.  We show, that the isomorphism class of a one-vertex  $\mathcal R$-graph can be recovered from the topological conjugacy class of its $\mathcal R$-graph shift. 
For certain  $\mathcal R$-graph semigroups of a one-vertex graph see 
 \cite [Section 4]{HK1}.
We also consider a class of examples of $\mathcal R$-graph shifts, to which the $\mathcal R$-graph semigroups of one-vertex $\mathcal R$-graphs are associated, and
that have the graph of a renewal system as an underlying graph
$G( \frak P, \{ (\frak q, \frak r) \in \frak P \times \frak P: \mathcal E^-(\frak q, \frak r) \neq \emptyset) \})$. For $\mathcal R$-graphs in this class we show, that the isomorphism class of the $\mathcal R$-graph can be recovered from the topological conjugacy class of its $\mathcal R$-graph shift. This result covers certain Markov-Dyck shifts. For other results on the reconstruction of a directed graph from its Markov-Dyck shift see \cite[Section 3]{KM1} and \cite{HK2}.

\section{  $\mathcal S_ \mathcal R(\frak P, \mathcal E^-\cup \mathcal E^+ ) $-presentations}

We denote the length of a directed path in a directed graph by $\ell$. 
Given an an $\mathcal R$-graph 
$$
G_{\mathcal R} = G_{\mathcal R}(\frak P, \mathcal E^-\cup \mathcal E^-),
$$ 
we denote by $\mathcal S^-(G_\mathcal R)$($\mathcal S^+(G_\mathcal R)$) the set of
non-zero elements of the subsemigroup of $\mathcal S^-(G_\mathcal R)$($\mathcal S^+(G_\mathcal R)$), that is generated by 
 $\mathcal E^-$($\mathcal E^+$).
There is the one-to-one
correspondence between 
the paths in the directed graph $G( \frak P, \mathcal E^-)$
( $G( \frak P, \mathcal E^+)$)
 and the elements of $\mathcal S^-(G_\mathcal R)$
 ($\mathcal S^+(G_\mathcal R)$).
We will use the same symbol to denote a path in
 $G( \frak P, \mathcal E^-)$
($G( \frak P, \mathcal E^+)$)
and the corresponding element of $\mathcal S^-(G_\mathcal R)$
($\mathcal S^+(G_\mathcal R)$)  (as we have already done for the edges in 
$ \mathcal E^- \cup \mathcal E^+  $). It will be clear from the context, which one is meant.
For the elements of $\mathcal S^-(G_\mathcal R)$
($\mathcal S^+(G_\mathcal R)$)
the  notations
$
\ell, 
 s, t$ are also used.
An element $g$ of $\mathcal S(G_\mathcal R)$ determines uniquely 
 $$
 \frak q(g) \in \frak P,  \quad
 u^+(g) \in  \{\bold 1_{\frak q(g)}  \}\cup \mathcal S^+(G_\mathcal R), \
 u^-(g) \in  \{\bold 1_{\frak q(g)}  \}\cup \mathcal S^-(G_\mathcal R),
 $$
 such that its normal form is given by
 $$
 g =  u^+(g) \bold 1_{\frak q(g)}u^-(g).
 $$
We write the normal forms of elements $g^-$ of $\mathcal S^-(G_\mathcal R)$, and 
of elements $g^+$ of $\mathcal S^+(G_\mathcal R)$ as
$$
g^- = \prod_{1 \leq i(-) \leq \ell(g^-)}e_{i(-)}^-[g(-)],
\quad
g^+ = \prod_{1 \leq i(+) \leq \ell(g^+)}e_{i(+)}^-[g(+)].
$$
We denote the set of non-zero elements of the subsemigroup of $\mathcal S
(G_\mathcal R )$, that is generated by $\mathcal E^-_\mathcal R$
($\mathcal E^-_\mathcal R$) by 
$\mathcal S^-_\mathcal R(G_\mathcal R)$($\mathcal S^+_\mathcal R(G_\mathcal R) $).

 \subsection{Context in 
 $\mathcal S(G_\mathcal R(\frak P, \mathcal E^-\cup  \mathcal E^+))$.}
 
In this subsection we consider an  $\mathcal R$-graph 
 $$
 G_\mathcal R = G_\mathcal R(\frak P, \mathcal E^-\cup  \mathcal E^+),
 $$
 such that
 \begin{align*}
\frak P^{(1)}\setminus { \frak P}_\mathcal  R ^{(1)}\neq \emptyset, \tag{2.1}%
\end{align*}
that satisfies conditions (I) and (II).
For  $f\in\mathcal S(G_\mathcal  R)$ we set
 $$
 \Gamma^-(f) = \{g \in  \mathcal S:gf \neq 0\}, \quad 
  \Gamma^+(f) = \{g \in  \mathcal S:fg \neq 0\},
 $$
 $$
 \Gamma (f) = \{(g(-), g(+) ) \in  \mathcal S^2: g(-)fg(+)  \neq 0   \},
 $$
 and we refer to $\Gamma (f)$ as the context of $f$.

We denote
for $\frak q, \frak r\in \frak P$  by $\ell_-(\frak q , \frak r )$ 
 ($\ell_+(\frak q , \frak r )  $) the length of a path in
  $G( \frak P , \mathcal E^- _\mathcal R)   $
 ($G( \frak P , \mathcal E^+ _\mathcal R ) $) from 
$\frak q$ to $\frak r$, provided such a path exists. By Condition (II) this notation is meaningful.

We denote for $\frak q \in { \frak P}_\mathcal  R ^{(1)}$ by $D_+(\frak q)$ the maximal length of a path in $G(\frak P,\mathcal E^+_\mathcal R )$ that leaves $\frak q$, and by
$D_-(\frak q)$ the maximal length of a path in $G(\frak P,\mathcal E^-_\mathcal R )$ that enters $\frak q$. We also set 

$$
D_\circ(\frak q) =
\begin{cases}
\min \{ D_+(\frak q) ,D_-(\frak q) \}, &\text{if  $\frak q \in { \frak P}_\mathcal  R ^{(1)}$}, 
\\
0, &\text {if  $\frak q \in \frak P^{(1)}\setminus { \frak P}_\mathcal  R ^{(1)}$}.
\end{cases}
$$
We set inductively
$$
\eta^k( \frak q ) = \eta(\eta^{k-1}(\frak q )), \qquad 1 < k < 
\max \{D_-(\frak q), D_+(\frak q)\},  \frak q \in \frak P^{(1)}.
$$
We remark, that  a path $b$ in $G(\frak P, \mathcal E^+_\mathcal R)$,  that starts at $\frak q$, and that has length  less then or equal to $D_\circ(\frak q)$, transverses the vertices 
$\eta^k, 1 \leq k  < \ell (b)$, before entering its target vertex. 
A similar remark applies to paths in $G(\frak P, \mathcal E^+_\mathcal R)$, that enter 
$\frak q$.

We set
$\eta^0(\frak q) =  \frak q$,
and for $\frak q \in \frak P^{(1)}$ we set
$$
\frak R_-(\frak q) = \{\eta^k(\frak q): 0 \leq k < D_-(\frak q)\}, 
\quad
\frak R_+(\frak q) = \{\eta^k(\frak q): 0 \leq k < D_+(\frak q)\},
$$
and for $ \frak q(-), \frak q(+) \in \frak P^{(1)}, $ such that
$$
 \frak q(-)\neq \frak q(+),
\quad
\frak R_+(\frak q(-)) \cap \frak R_-(\frak q(+)) \neq \emptyset,
$$
we denote by  $H_+( \frak q(-) ,  \frak q(+) )$($H_-( \frak q(-) ,  \frak q(+) ) $) the minimal length of a path in $G(\frak P, \mathcal E^+_\mathcal R)$
($G(\frak P, \mathcal E^-_\mathcal R)$), that has $\frak q(-)$($\frak q(+)$) as source vertex and a vertex in 
$\frak R_+(\frak q(+)))$ ($\frak R_+(\frak q(-)))$) as target vertex.

\begin{lemma}
For $ \frak q(-), \frak q(+) \in \frak P^{(1)}, $ such that
$$
 \frak q(-)\neq \frak q(+),
\quad
\frak R_+(\frak q(-)) \cap \frak R_-(\frak q(+)) \neq \emptyset,
$$
one has that
\begin{align*}
\eta^{H_+(\frak q(-),\frak q(+))}(\frak q(-)) = \eta^{H_-(\frak q(-),\frak q(+))}(\frak q(+)). 
\tag {2.2}
\end{align*}
\end{lemma}
\begin{proof} 
The inequality
$$
\eta^{H_+(\frak q(-),\frak q(+))}(\frak q(-)) \neq \eta^{H_-(\frak q(-),\frak q(+))}(\frak q(+)),
$$
would imply  the existence of a path in $G( \frak P , \mathcal E^+  )$, as well as in 
$G( \frak P , \mathcal E^+  )$, from 
$\eta^{H_+(\frak q(-),\frak q(+))}(\frak q(-))  $ to
 $\eta^{H_-(\frak q(-),\frak q(+))}(\frak q(+))  $, 
 and also the existence of a path in $G( \frak P , \mathcal E^+  )$, as well as in 
$G( \frak P , \mathcal E^+  )$, from 
$\eta^{H_-(\frak q(-),\frak q(+))}(\frak q(+)) $ to 
$\eta^{H_+(\frak q(-),\frak q(+))}(\frak q(-))  $, contradicting the assumption (2.1).
\end{proof} 
We denote the vertex, that appears in (2.2) by $\frak p(\frak q(-),\frak q(+))$.

\begin{Lemma A}
Let $\frak q \in \frak P$. For 
\begin{align*}
f^+ \in  \{ \bold 1_ {\frak q} \}\cup\mathcal S^+( G_\mathcal R ), \quad f^- \in \{ \bold 1_ {\frak q} \}\cup
 \mathcal S^-( G_\mathcal R  ), \tag{2.3}
\end{align*}
such that
\begin{align*}
s(f^+) =\bold 1_ {\frak q} = t(f^-), \quad t(f^+ ) = s(f^-), \tag{2.4}
\end{align*}
all elements of $\mathcal S(G_\mathcal R)$ of the form $f^+f^-$ have the same context.
\end{Lemma A}

\begin{proof} 
One notes, that
$$
0 \leq \ell (f^+) = \ell (f^-) \leq D_\circ(\frak q ).
$$
Set
\begin{align*}
&\Gamma^-_\circ(\frak q) = \{g(-) \in 
 \Gamma^-(\frak q): \ell(h^-(g(-)))  \leq D_\circ (\frak q)\},
\\
&\Gamma^+_\circ(\frak q) = \{g(+) \in  \Gamma^+(\frak q): \ell(h^+(g(-)))
  \leq D_\circ (\frak q)\},
\end{align*}
\begin{align*}
\Gamma_\circ(\frak q) = \{( g(-) , g(+) ) \in (&\Gamma^-(\frak q)\setminus  \Gamma^-_\circ(\frak q)   ) \times ( \Gamma^+(\frak q)\setminus  \Gamma^+_\circ(\frak q)    ): 
\\
(\prod_{1 \leq i(+) <\ell( h^- (g(-)) ) - D_\circ (\frak q))}&e^+_{i(+)}[g(-)])
\ \bold 1_{\eta^{D_\circ(\frak q)}(\frak q)}
\\
&(\prod_{\ell( h^+(g(+)) - D_\circ (\frak q) < i(+) \leq\ell( h^+(g(+))}e^+_{i(+)}[g(-)]) \neq 0.
\end{align*}

By Condition (II) one has for $f^+,f^-$ as in (2.3) and (2.4)
$$
\Gamma(f^+f^-) = ( \Gamma^-_\circ(\frak q) \times \Gamma^+(\frak q)) 
\cup (\Gamma^-(\frak q) \times \Gamma^+_\circ(\frak q) )  \cup \Gamma_\circ(\frak q).\qed
$$
\renewcommand{\qedsymbol}{}
\end{proof} 

\begin{Lemma B} 
Let $ \frak q(-), \frak q(+) \in \frak P^{(1)}, $ such that
$$
 \frak q(-)\neq \frak q(+).
 $$
\text{(B$-$)}   Let there exist a path in $S^+( G_\mathcal R  )$ from $\frak q(-)$ to $\frak q(+)$. Then for 
 \begin{align*}
 h^+ \in \mathcal S( G_\mathcal R  ), \quad f^+ \in \{ \bold 1_{\frak q(+)} \} \cup 
 \mathcal S^+( G_\mathcal R  ), \quad f^- \in 
\{ \bold 1_{\frak q(+)} \} \cup \mathcal S^-( G_\mathcal R  ), \tag{2.5}
 \end{align*}
 such that
 \begin{align*}
 s( h^+) = \bold 1_{ \frak q(-)}, \quad t( h^+) =  \bold 1_{ \frak q(+)}, \tag{2.6}
 \end{align*}
 \begin{align*}
 s(f^+ ) =\bold 1_{ \frak q(-)}, \quad  t(f^+) =  s(f^-), \quad t(f^-) =  \bold 1_{ \frak q(+)},
  \end{align*}
 all elements of $\mathcal S(G_\mathcal R)$ of the form 
  $ h^+ f^+f^-$ have the same context.
  
  \medskip
  \noindent
  (B+)  Let there exist a path in from $\frak q(+)$ to $\frak q(-)$. Then for 
 \begin{align*}
 h^+ \in \mathcal S (G_\mathcal R  ), \quad f^+ \in \{ \bold 1_{\frak q(+)} \} \cup \mathcal S^+( G  ), \quad f^- \in 
\{ \bold 1_{\frak q(+)} \} \cup \mathcal S^-( G_\mathcal R  ), 
 \end{align*}
 such that
 \begin{align*}
 &s( h^+) = \bold 1_{ \frak q(-)}, \quad t( h^+) =  \bold 1_{ \frak q(+)}, 
 \\
 s(f^+ ) =  
& \bold 1_{ \frak q(-)},
  \quad  t(f^+) =  s(f^-), 
 \quad t(f^-) =  \bold 1_{ \frak q(+)},
  \end{align*}
 all elements of $\mathcal S(G_\mathcal R)$ of the form
  $  f^+f^-h^-$ have the same context.
\end{Lemma B} 
\begin{proof} 
We prove (B-).
We note that
$$
0 \leq \ell (f^+) = \ell (f^-) \leq D_\circ(\frak q(+) ).
$$
Set
\begin{align*}
&\Gamma^-_\circ(\frak q(-), \frak q(+)) = \{g(-) \in  \Gamma^-(\bold 1_{ \frak q(-)}): \ell(h^+(g(-))) 
 \leq  \ell(  \frak q(-),  \frak q(+) ) +D_\circ (\frak q(+))\},
\\
&\Gamma^+_\circ(\frak q(+)) = \{g(+) \in  \Gamma^+(\bold 1_{ \frak q(+)}): \ell(h^-(g(+))) 
 \leq D_\circ (\frak q(+))\},
\end{align*}
\begin{multline*}
\Gamma_\circ(\frak q(-), \frak q(+)) = 
\\
( g(-) , g(+) ) \in 
(\Gamma^-(\frak q(-))\setminus  
\Gamma^-_\circ(\frak q(-), \frak q(+) ) 
\times ( \Gamma^+(\frak q(+))\setminus 
 \Gamma^+_\circ(\frak q(+))    ): 
\end{multline*}
\begin{multline*}
(\prod_{1 \leq i(-)< \ell( h^- (g(-)) ) - \ell_+(\frak q (-) , \frak q(+))-D_\circ (\frak q(+))}
e^-_{i(-)}[g(-)])
\ \bold 1_{\eta^{D_\circ(\frak q(+))}(\frak q(+))}
\\
(\prod_{D_\circ (\frak q(+)) < i(+) \leq\ell( h^+(g(+))}e^+_{i(+)}[g(+)]) \neq 0.
\end{multline*}

By Condition (II) one has  for $h^+,  f^+,  f^-,$ as in (2.5) and (2.6), that
\begin{align*}
&\Gamma(h^+f^+f^-) =
\\
&( \Gamma^-_\circ(\frak q(-),\frak q(+)) \times \Gamma^+(\frak q(+)) 
\cup (\Gamma^-(\frak q(-)) \times \Gamma^+_\circ(\frak q(+)) )  \cup 
\Gamma_\circ(\frak q(-)\frak ,q(+)).
\end{align*}
The proof of (B+) is symmetric.
\end{proof}

\begin{Lemma C}
Let $ \frak q(-), \frak q(+) \in \frak P^{(1)}, $ such that
\begin{align*}
\frak q(-)\neq \frak q(+), 
 \end{align*}
 and
 $$
 \frak R(\frak q(-)) \cap  \frak R(\frak q(+)) \neq \emptyset,
 \quad
 \frak p(\frak q(-), \frak q(+))\notin \{\frak q(-), \frak q(+) \}.
 $$
 Then for
 \begin{align*}
 h^+, f^+ \in \mathcal S^+(G_\mathcal R),  \ \  f^-, h^- \in \mathcal S^-(G_\mathcal R),  
 \tag{2.7}
 \end{align*}
 such that
 \begin{align*}
 s(h^+) = \bold 1_{ \frak q(-)},\  t(h^+) = \bold 1_{\frak p( \frak q(-),\frak q(+) )},   \tag{2.8}%
 \end{align*}
 \begin{align*}
 s(f^+ ) =   \bold 1_{\frak p( \frak q(-),\frak q(+) )}, \ t(f^+)  = s(f^-), \ t(f^-) = 
 \bold 1_{\frak p( \frak q(-),\frak q(+) )} ,
 \end{align*}
 \begin{align*}
 s(h^-  ) =  \bold 1_{\frak p( \frak q(-),\frak q(+) )} , \  t( h^- ) = \bold 1_{\frak q(+)},
 \end{align*}
all elements of $\mathcal S(G_\mathcal R)$ of the form  $h^+f^+f^-h^- $ have the same context.
\end{Lemma C} 
\begin{proof} 
One notes that
$$
0\leq \ell(f^+) = \ell(f^-)  \leq D_\circ(\frak p( \frak q(-),\frak q(+) ) ).
$$

Set
\begin{multline*}
\Gamma^-_\circ(\frak q(-), \frak q(+)) = 
\\
\{g(-) \in  \Gamma^-(\bold 1_{\frak q(-)}): \ell(h^-(g(-))) 
 \leq  H_-(\frak q(-),  \frak q(+) ) +D_\circ(\frak p( \frak q(-),\frak q(+) ))\},
\end{multline*}
\begin{multline*}
\Gamma^-_\circ(\frak q(-), \frak q(+)) = 
\\
\{g(+) \in  \Gamma^+(\bold 1_{\frak q(+)}): \ell(h^+(g(+))) 
 \leq  H_+(\frak q(-),  \frak q(+) ) +D_\circ (\frak p( \frak q(-),\frak q(+) ))\},
\end{multline*}
\begin{align*}
&\Gamma_\circ(\frak q(-),\frak q(+) ) = 
\\
&\{( g(-) , g(+) ) \in (\Gamma^-(\bold 1_{\frak q(-)})\setminus  \Gamma^-_\circ(\frak q(-), \frak q(+) ) \times ( \Gamma^+(\bold 1_{\frak q(+)})\setminus 
 \Gamma^+_\circ(\frak q(+),\frak q(+))    ): 
\end{align*}
\begin{multline*}
(\prod_{1 \leq i(-)< \ell( h^- (g(-)) ) - \ell_+(\frak q (-) , \frak q(+))-D_\circ (\frak q(+))}
e^-_{i(-)}[g(-)])
\ \bold 1_{\eta^{D_\circ(\frak p( \frak q(-),\frak q(+) ))}(\frak p( \frak q(-),\frak q(+))}
\\
(\prod_{D_\circ (\frak q(+)) < i(+) \leq\ell( h^+(g(+))}e^+_{i(+)}[g(+)]) \neq 0.
\end{multline*}
By Condition (II) one has for  that for $h^+, f^+ ,  f^-, h^- $ as in (2.7), (2.8), that
\begin{align*}
&\Gamma(h^+f^+f^-h^-) =
\\
&( \Gamma^-_\circ(\frak q(-),\frak q(+)) \times \Gamma^+(\bold 1_{\frak q(+)}) 
\cup (\Gamma^-(\bold 1_{\frak q(-)}) \times \Gamma^+_\circ(\frak q(-),\frak q(+)) )  \cup 
\Gamma_\circ(\frak q(-)\frak ,q(+)).\qed
\end{align*}
\renewcommand{\qedsymbol}{}
\end{proof} 
\subsection{Property (A) and $\mathcal R$-graph shifts.}

We introduce  notation and terminology for subshifts. The set of periodic points of a subshift $X$ we denote by $P(X)$. The smallest period of $p \in P(X)$, we denote by 
$\pi(p)$.
We denote the language of admissible words of a subshift $X\subset\Sigma^{\Bbb Z}$ by 
$\mathcal L(X)$. The context of a word $b\in \mathcal L(X)$ is defined as the set
$$
\Gamma(b) = \{c(-), c(+)) \in  \mathcal L(X)^2: C(-)bc(+) \in \mathcal L(X)\}.
$$
Concerning $\mathcal S(G(\frak P, \mathcal E^- \cup \mathcal E^+  ))$-presentations
 $X(\mathcal V, \Sigma, \lambda))$ we remark, that 
\begin{multline*}
\Gamma (b) = 
\{ (c(-), c(+)) \in \mathcal L(X(\mathcal V, \Sigma, \lambda)))^2: 
t(c(-)) = s(b), t(b) = s(c(+)),
\\
(\lambda(c(-)), \lambda(c(+))) \in \Gamma(\lambda(b))\},  \quad
b \in \mathcal L(X(\mathcal V, \Sigma, \lambda))).
\end{multline*}

Given a subshift $X \subset \Sigma^{\Bbb Z}$ we set
$$
	x_{[i,k]}Ê=Ê(x_{j})_{i \leq j \leq k}, \quad  		x \in X,  
			i,k\in \Bbb Z ,  i \leq k,
$$
and
$$
X_{[i,k]}Ê=Ê\{ x_{[i,k]} : x\in X \},i,k\in \Bbb Z ,  i \leq k . 
$$

We set
\begin{multline*}
\Gamma(a)Ê= \bigcup_{n,m \in \Bbb N}
\{(b,c) \in  X_{[i-n, i]} \times X_{[k, k+m]}: (  b , a , c)  \in X_{[i-n, k+ m]}\}, \\ a\in X_{[i,k]},	 i,k\in \Bbb Z ,  i \leq k.	
\end{multline*}
and call $\Gamma(a)$ the context of the block $a$.
We set
\begin{multline*}
\Gamma^{+}_n(a)Ê=\{ b\in X_{(k,k+n]}: (a,b)\in X_{[i,k+n]}\} , \\
 n \in \Bbb N,a\in X_{[i,k]},	 i,k\in \Bbb Z ,  	i \leq k.
 \end{multline*}
$\Gamma^-$ has the symmetric meaning.
We set
\begin{multline*}
 \omega^{+}(a)Ê= \bigcup_{n\in \Bbb N}\bigcap_{n\in \Bbb N}\bigcap_{c \in \Gamma_{n}^{-}(a)}
\{  b  \in X_{(k,k+n]}: (c, a,  b ) \in X_{_{[i-n, k+ m]}}\},Ê\\
 a\in X_{[i,k]},	 i,k\in \Bbb Z ,  	i \leq k.
\end{multline*}
$\omega^{-}$ has the  symmetric meaning. 

Given a subshift $X\subset \Sigma ^\Bbb Z$ we define a subshift of finite type (more precicely, an  n-step Markov shift) $A_{n}(X)$ by 
$$
 A_{n}(X) = \bigcap _{i\in \Bbb Z}
( \{x  \in X: 
x_{i} \in \omega^{+} (x_{[i - n, i)})\}\cap( \{x  \in X: 
x_{i} \in \omega^{-} (x_{(i,i + n]})\})\,\quad n \in \Bbb N,
$$
and we set
$$
A(X) = \bigcup_{n \in \Bbb N} A_{n}(X).
$$

We recall from \cite {Kr2} the definition of Property (A).
For $n \in \Bbb N$ a subshift $X \subset  \Sigma ^{\Bbb 
Z},$ 
has property $(a, n, H),H \in \Bbb N,$  if for $h,
\widetilde {h} \geq 3H$ and for
$
I_-,  I_+,\widetilde{I}_-,  \widetilde{I}_+ \in \Bbb Z,
$
such that
$$
I_+-I_-, \widetilde{I}_+-\widetilde{I}_- \geq 3H,
$$
and for
$$
a \in A_{n}(X)_{(I_-,  I_+ ]}, \quad \widetilde{a} \in A_{n}(X)_{(\widetilde{I}_-, 
 \widetilde{I}_+ ]}, 
$$
such that
$$
a_{(I_-, I_- +  H]} = \tilde{a}_{(\widetilde{I}_-, \widetilde{I}_ + H]},\quad
a_{(  I_+  - H,   I_+ ]} = \tilde{a}_{( \widetilde{I}_+ - H,  \widetilde{I}_+]},
$$
one has that $a$ and $\tilde{a}$ have the same context. A subshift $X \subset  \Sigma ^{\Bbb Z}$
has property $(A)$ if there are $H_{n}, n \in  \Bbb N$, such
that $X$ has the properties $(a, n, H_{n}), n \in \Bbb N $.

\bigskip
\begin{theorem}
 Let 
 $$
 G_ \mathcal R =  G_ \mathcal R(\frak P, \mathcal E^-\cup \mathcal E^+ ) ,
 $$
 be an $\mathcal R$-graph such that 
 \begin{align*}
 \frak P^{(1)} \setminus\frak P^{(1)}_\mathcal R  \neq \emptyset. \tag{2.9}
\end{align*}
For an $\mathcal S_ \mathcal R(\frak P, \mathcal E^-\cup \mathcal E^+ ) $-presentation
$X(\mathcal V, \Sigma, \lambda)$   to have Property (A) it is necessary and sufficient, that 
$G_ \mathcal R(\frak P, \mathcal E^-\cup \mathcal E^+ ) $ satisfies Conditions (I) and (II).
\end{theorem}
\begin{proof}
We prove necessity.
We choose vertices 
$
V_\frak p\in \mathcal V( \frak p ),  \frak p \in \frak P ,
$
and cycles $c_\frak p, \frak p \in \frak P $ in the graph $G(\mathcal V, \Sigma)$, such that
$$
s(c_\frak p ) =  t(c_\frak p ) = V_\frak p, \quad  \lambda(c_\frak p ) = \bold 1_\frak p ,
 \qquad \frak p \in \frak P. 
$$
For $ k \in \Bbb N$ we denote by $c_\frak p^k $ the cycle, that transverses k-times the 
cycle $c_\frak p$. Also we choose for all $e^- \in \mathcal E^-$ a path $a_{e^-}$ in the graph  
$G(\mathcal V, \Sigma)$, such that
$$
s(a_{e^-})  =   V_{s(e^-)}   ,   t(a_{e^-})  =   V_{t(e^-)}, \quad  \lambda(a_{e^-} ) = e^-. 
$$
and we make similar choices for all $e^+ \in \mathcal E^+$. We set
$$
M = \max(\{\ell(a_{e^-}): e^- \in \mathcal E^-  \} \cup   \{\ell(c_\frak p): \frak p \in \frak P \} \cup  \{\ell(a_{e^+}):e^+ \in \mathcal E^+\}).
$$
One has that
\begin{align*}
c_\frak p^k   \in    \mathcal L( A_{M}(X(\mathcal V, \Sigma, \lambda))), \qquad k \in \Bbb N, \frak p \in \frak P.  
\end{align*}

\noindent
(I). Assume that the $\mathcal R$-graph 
$\mathcal G_ \mathcal R(\frak P, \mathcal E^-\cup \mathcal E^+ ) $ does not satisfy Condition (I). Under this assumption we can choose  a vertex $\frak p \in \frak P^{1} \setminus \frak P^{1}_\mathcal R$, and edges
\begin{align*}
e^- \in   \mathcal E^-_\mathcal R(\eta(\frak p), \frak p ), \qquad 
e^+ \in   \mathcal E^+_\mathcal R(\eta(\frak p), \frak p ), \tag {2.10}%
\end{align*}
such that
\begin{align*}
(e^-  , e^+ ) \notin \mathcal R(\eta(\frak p), \frak p ) . \tag{2.11}%
\end{align*}
By construction
\begin{align*}
(c_\frak p^k,  a_{e^+} , a_{e^-}, c_\frak p^k )\in \mathcal L( A_{M}(X(\mathcal V, \Sigma, \lambda))), \quad k \in \Bbb N. 
\end{align*}
It follows from (2.10), that
$$
(a_{e^-} , a_{e^+}) \in \Gamma(c_\frak p^k,  a_{e^+} , a_{e^-}, c_\frak p^k ), \quad 
k \in \Bbb N.
$$
and it follows from (2.11), that
$$
(a_{e^-} , a_{e^+}) \notin \Gamma ( c_\frak p^{2k}  ), \quad k \in \Bbb N
$$
We have shown, that $X(G( \mathcal V , \Sigma  , \lambda) ) $ does not have Property (A).

\noindent
(II) Assume that the $\mathcal R$-graph $G_ \mathcal R(\frak P, \mathcal E^-\cup \mathcal E^+ ) $ does not satisfy Condition (II -). By (2.9) every cycle in 
$G( \frak P, \mathcal E^-)$ transverses at least one vertex in 
$\frak P^{(1)} \setminus \frak P^{(1)}_\mathcal R$.
We can therefore choose a path 
$$
f^- = (e^-_l)_{1 \leq l \leq \ell ( f^- )},
$$
in the graph $G(\frak P,\mathcal E^-_\mathcal R  )$, and a vertex
$\widetilde {\frak q} \in  \frak P^{(1)} \setminus \frak P^{(1)}_\mathcal R$, such that
$$
s(f^-) = \widetilde {\frak q} = t(f^- ),
$$
together with edges
\begin{align*}
\widetilde{e}^- \in \mathcal E^-_\mathcal R, \qquad  \widetilde{e}^+ \in 
\mathcal E^+_\mathcal R, \tag{2.12}%
\end{align*}
such that
$$
t(\widetilde{e}^-) = \widetilde {\frak q} = s(\widetilde{e}^+ ),
$$
and such that
\begin{align*}
(\widetilde{e}^- ,\widetilde{e}^+ ) \notin \mathcal R. \tag{2.13}
\end{align*}
We set
$$
a_{ f^-} = (a_{e^-_l})_{1 \leq l \leq \ell ( f^- )}.
$$
By construction
\begin{align*}
(c_\frak p^k,  a_{ f^-}, c_\frak p^k )\in \mathcal L( A_{M}(X(G))), \quad k \in \Bbb N. 
\end{align*}
It follows from (2.12), that
$$
(a_{e^-} , a_{e^+}) \in \Gamma(c_\frak p^k,  a_{e^+} , a_{e^-}, c_\frak p^k ), \quad 
k \in \Bbb N.
$$
and it follows from (2.13), that
$$
(a_{e^-} , a_{e^+}) \notin \Gamma ( c_\frak p^{2k}  ), \quad k \in \Bbb N
$$
We have shown, that $X(G( \mathcal V , \Sigma  , \lambda) ) $ does not have Property (A).

\noindent
The case (II +) is symmetric.

We prove sufficiency.
Let $n \in \Bbb N$, and let 
 $a(-), a(+)\in \mathcal L(X(\mathcal V, \Sigma, \lambda ) ), \ell(a(-)), \ell (a(+))  = n.$
 For $m\in[1,n]$ denote by by $a(-)_{[1,m]}$($a(+)_{[m,n]} $) the prefix (suffix) of length 
 $m$ ($n - m + 1  $) of
  $a(-)$($a(+)$). Set
 
 \begin{multline*}
M(-) =
\\
\begin{cases}
\max \{m \in [1, n]: \lambda (a(-)_{[1, m]}) = u^+(\lambda (a(-))  \}
&\text{if  $u^+(\lambda (a(-)) \in \mathcal S^+(G_ \mathcal R)$}, \\
0, &\text {if  $u^+(\lambda (a(-)) = \bold 1_{\frak q(\lambda (a(-))}$},
\end{cases}
\end{multline*}

\begin{multline*}
M(+) =
\\
\begin{cases}
\max \{m \in [1, n]: \lambda (a(+)_{[n - m, n]}) = u^-(\lambda (a(-))  \}
&\text{if  $u^-(\lambda (a(+)) \in \mathcal S^-(G_ \mathcal R)$}, \\
0, &\text {if  $u^-(\lambda (a(+)) = \bold 1_{\frak q(\lambda (a(+))}$}.
\end{cases}
\end{multline*}
Let
$$
K(-), K(+), \bar K(-), \bar K(+) \in \Bbb Z, 
$$
$$
K(+)  - K(-) ,  \bar K(+) - \bar K(-) >  2n,
$$
and let
\begin{align*}
b \in A_n(X(G))_{[K(-), K(+)]},  \quad \bar b \in A_n(X(G))_{[\bar K(-),\bar K(+)]},  \tag{2.14}
\end{align*}
be such that
$$
b_{[K(-), K(-) + n ]} = 
\bar b_{[\bar K(-),\bar K(-) + n ]} = a(-),
$$
$$
b_{[K(+) - n, K(+) ]} =
 \bar b_{[\bar K(+) - n,\bar K(+) ]} = a(+),
$$
Set
$$
c = b_{[K(-) + M(-), K(+) - M(+)]}, \quad \bar c =\bar b_{[\bar K(-) + M(-), \bar K(+) - M(+)]}.
$$
We consider the four cases (A),  (B-),  (B+)  and (C):

\noindent
(A) In the case, that $\frak q(\lambda(a(-)))  $ = $\frak q(a(-))$, let 
set
$$
f^+ = h^+(\lambda (c)), f^- = h^-(\lambda (c)), \quad \bar f^+ = \bar  h^+(\lambda (\bar c)),
\bar f^- = \bar h^-(\lambda (\bar c)),
$$
It follows from (2.14), that
$$
f^+, \bar  f^+ \in \mathcal S^+_\mathcal R(G_ \mathcal R), \quad
 f^-, \bar  f^- \in \mathcal S^-_\mathcal R(G_ \mathcal R).
$$
By Lemma A
\begin{align*}
\Gamma( f^+f^- ) = \Gamma( \bar  f^+\bar  f^- ). \tag {2.15}%
\end{align*}

\noindent
(B -) In the case, that 
$$
\frak q(\lambda(a(-)))  \neq\frak q(\lambda (a(+)),  \quad\frak q(\lambda(a(-)))= 
\frak p(\frak q(\lambda(a(-))), \frak q(\lambda(a(+))),
$$
one has
$$
\ell(h^+(\lambda(c)), \ell(h^+(\lambda(c)) \geq \ell_+( \frak q(\lambda (a(-)) ,  \frak q(\lambda (a(+))   ).
$$
Let
$$
h^+, \ \   f^+,, f^-,  \quad \bar h^+, \ \  \bar f^+,,\bar  f^-,
$$
be given by
$$
\ell(h^+) = \ell( \bar h^+)  = \ell_+(\frak q(\lambda(a(-)), \frak q(\lambda(a(+)),
$$
$$
h^+ = u^+(\lambda(c)) f^+, \quad \bar h^+ =   u^+(\lambda(\bar c)) f^+.
$$
By (2.14)
$$
h^+ , \bar h^+,   f^+ , \bar f^+\in \mathcal S^-_\mathcal R(G_ \mathcal R), \quad 
f^- , \bar f^- \in \mathcal S^+_\mathcal R(G_ \mathcal R).
$$
By Lemma B,
\begin{align*}
\Gamma ( h^+f^+f^-  ) = \Gamma (\bar h^+\bar f^+\bar f^-  ). \tag {2.16}%
\end{align*}
(C): In the case, that 
$$
\frak q(\lambda(a(-)))  \neq\frak q(\lambda (a(+)),  
$$
$$
\frak p(\frak q(\lambda(a(-))), \frak q(\lambda(a(+))) \notin
\{\frak q(\lambda(a(-))),\frak q(\lambda(a(+)))  \},
$$
Let
$$
h^+ \bar h^+ ,f^+,\bar f^+, \in \mathcal S^+_\mathcal R(G_\mathcal R), \quad
f^-,\bar f^-,h^- \bar h^- ,\in \mathcal S^-_\mathcal R(G_\mathcal R),
$$
be given by
$$
\ell(  h^+) = \ell( \bar h^+ ) = \ell_+( \frak q(\lambda(a(-))), 
\frak p( \frak q(\lambda(a(-))) , \frak q(\lambda(a(+))))) ,
$$
$$
\ell(  h^-) = \ell( \bar h^- ) = \ell_-(\frak p(  \frak q(\lambda(a(-)))), \frak q(\lambda(a(+)))
, \frak q(\lambda(a(+)))) ,
$$
$$
u^+(c) =h^+f^+,  \ u^+(\bar c) =\bar h^+\bar f^+, \quad
u^-(c) =h^-f^-,  \ u^-(\bar c) =\bar h^-\bar f^-.
$$
By (2.14)
$$
h^+ , \bar h^+,   f^+ , \bar f^+\in \mathcal S^-_\mathcal R(G_ \mathcal R), \quad 
f^- , \bar f^- ,h^+ , \bar h^+\in \mathcal S^+_\mathcal R(G_ \mathcal R).
$$
By Lemma C,
\begin{align*}
\Gamma ( h^+f^+f^- h^- ) = \Gamma (\bar h^+\bar f^+\bar f^- \bar h^- ). \tag {2.17}
\end{align*}

It follows from (2.15 - 17), that in all cases the context of $b$ is equal to the context or $\bar b$.
It is shown, that $X$ has properties $a(n,n), n \in \Bbb N$.
\end{proof}

\section{The  $\mathcal R$-graph semigroup associated to an \\  
$\mathcal S(G_\mathcal R(\mathcal P, \mathcal E^-\cup \mathcal E^+ ))$-presentation }

Following the terminology, that was introduced in \cite {HI}, we say for an $\mathcal R$-graph
$ G_\mathcal R(\frak P, \mathcal E^-\cup\mathcal E^+)$
 and an 
$\mathcal S(G_\mathcal R(\frak P, \mathcal E^-, \mathcal E^+))$-presentation $X(\mathcal V,\Sigma, \lambda)$,  that a periodic point $p$ in 
$ X(\mathcal V,\Sigma, \lambda) )$ is neutral, if there exist $I \in \Bbb Z$ and $\frak p \in \frak P$ such that $
\lambda (p_{[I, I +\pi(p) } ) = \bold 1_\frak p,
$
and we say that a periodic point $p$ in $ X(\mathcal V,\Sigma, \lambda) $ has negative (positive) multiplier, if there exist $I \in \Bbb Z$,  such that 
$\lambda (p_{[I, I +\pi(p) } ) \in \mathcal S^-(\frak P, \mathcal E^+\cup \mathcal E^-) (\mathcal S^+(\frak P, \mathcal E^+\cup \mathcal E^-)) .
$
\begin{lemma}
Let $ G_\mathcal R(\frak P, \mathcal E^-\cup\mathcal E^+)$ 
  be an $\mathcal R$-graph, such that
$
\frak P \setminus{\frak P}_\mathcal R^{(1)} \neq \emptyset, 
$
that satisfies Conditions  (II), and let 
$X(\mathcal V,\Sigma, \lambda)$ 
be an
$\mathcal S(G_\mathcal R(\frak P, \mathcal E^-, \mathcal E^+))$-presentation. \ 
Then  a periodic point of $X(\mathcal V,\Sigma, \lambda)$ is in $A(X(\mathcal V,\Sigma, \lambda))$ if and only if it is neutral.
\end{lemma}
\begin{proof}
Let
 \begin{align*}
 p\in A( X(\mathcal V,\Sigma, \lambda)), \tag {3.1}
 \end{align*}
  and let $I \in \Bbb Z$ be such that
$$
\lambda (p_{[I, I +\pi(p) } )\in \mathcal S^-(\frak P, \mathcal E^-\cup \mathcal E^+) \cup \{ \bold 1_\frak p: \frak p \in \frak P \} \cup \mathcal S^+(\frak P, \mathcal E^-\cup \mathcal E^+).  
$$
If
$$
\lambda (p_{[I, I +\pi(p) } )\in \mathcal S^-_\mathcal R(\frak P, \mathcal E^-\cup \mathcal E^+) ,
$$
then it follows from (3.1) that $\lambda (p_{[I, I +\pi(p) } )$ is given by a cycle in the directed graph $G(\frak P, \mathcal E^-_\mathcal R)$  ,  contradicting Condition (II-). For the case that
$$
\lambda (p_{[I, I +\pi(p) } )\in \mathcal S^+(\frak P, \mathcal E^-\cup \mathcal E^+) ,
$$
one has the symmetric argument.

For the converse, note that
$
\lambda (p_{[I, I +\pi(p)) } ) = \bold 1_\frak p,
$
implies
$
p \in A_{\pi(p)}(X(\mathcal V,\Sigma, \lambda)).
$
\end{proof}

Given finite sets $\mathcal E^-$ and $\mathcal E^-$ and a relation $\mathcal R \subset \mathcal E^-  \times  \mathcal E^+  $, we say that $e^- \in  \mathcal E^-$ and 
 $\tilde e^- \in  \mathcal E^-$
 are $\sim(\mathcal R, -)$-equivalent if
 $$
 \{\tilde e^+ \in \mathcal E^+: (e^-, \tilde e^+ ) \in \mathcal R    \} =
 \{\tilde e^+ \in \mathcal E^+: (\tilde e^-, \tilde e^+ ) \in \mathcal R    \}.
 $$
An equivalence relation  $\sim(\mathcal R, +)$ on $\mathcal E^+$ is defined symmetrically. 
Given 
 an $\mathcal R$-graph $\mathcal G_\mathcal R(\frak P, \mathcal E^-\cup \mathcal E^+)$ 
we we construct  an $\mathcal R$-graph 
$\mathcal G_{\bar{\mathcal R}}(\frak P, \bar{\mathcal E}^-\cup \bar{\mathcal E}^+)$,  by setting
\begin{align*}
&\bar{\mathcal E}^-(\frak q, \frak r) = [\mathcal E^- (\frak q, \frak r)]_{\sim(\mathcal R , -)},
\ \ \ \ \
\bar{\mathcal E}^+(\frak q, \frak r) = [\mathcal E^- (\frak q, \frak r)]_{\sim(\mathcal R , +)},
\\
&\bar{\mathcal R}(\frak q, \frak r) = \{ (\bar e^-  , \bar e^+ )\in 
\bar{\mathcal E}^-(\frak q, \frak r) \times \bar{\mathcal E}^+(\frak q, \frak r) : 
\bar{e}^- \times \bar{ e}^+ 
\subset\mathcal R(\frak q, \frak r) \}, \qquad 
\frak q, \frak r \in \frak P.
\end{align*}

We denote by
$ \bar {\mathcal F}^-_{\bar {\mathcal R}}$ the set of edges in $\bar {\mathcal E}^-$, that are the single incoming edges of their target vertices, and we denote by 
$ \bar {\mathcal F}^+_{\bar {\mathcal R}}$ the set of edges in $\bar {\mathcal E}^+$, that are the eingle outgoing edges of their source vertices. Observe that the set $\frak P^{(1)}_\mathcal R$ is  the set of target vertices of the edges in 
$\bar {\mathcal F}^-_{\bar {\mathcal R}}$, which is
equal to the set of source vertices of the edges in 
$\bar {\mathcal F}^+_{\bar {\mathcal R}}$, 
and that $\frak P \setminus \frak P^{(1)}_\mathcal R$ is the set of vertices,  that have at 
least two incoming edges in
 $ \bar {\mathcal E}^-$,
 or, equivalently, that have at least two outgoing edges in $ \bar {\mathcal E}^+$.
 We set
 $$
 \widehat{\frak R} = \frak P \setminus \frak P^{(1)}_\mathcal R,
$$
and we denote 
by 
$\widehat{\frak R}^\bullet$ the set of vertices in $\widehat{\frak R}$, that are source vertices of an edge in  $ \bar {\mathcal F}^-_{\bar {\mathcal R}}$, or, equivalently, that are target vertices of an edge in  
$ \bar {\mathcal F}^+_{\bar {\mathcal R}}$.
For $\frak r \in \frak R_{\bar {\mathcal R}}^\bullet$ we denote by
 $\frak P(\frak r)$ the set of $\frak p \in \frak P$, that are target vertices of a path in the graph $G( \frak P  , \bar {\mathcal F}^-_{\bar {\mathcal R}} )$, that leaves 
$\frak r$, or, equivalently, that are source vertices of a path in the graph 
$G( \frak P  , \bar {\mathcal F}^+_{\bar {\mathcal R}} )$, that enters 
$\frak r$. For 
$\frak r \in  \widehat{\frak R}\setminus\widehat{\frak R}^\bullet$
 we set
$
\frak P(\frak r) =\{ \frak r \}.
$
 We also set
 \begin{align*}
 &
 \bar {\mathcal F}^-_{\bar {\mathcal R}}(\frak r) =
 \{ f^- \in  \bar {\mathcal F}^-_{\bar {\mathcal R}}: t( f^- ) \in   \frak P(\frak r)   \},
\\
&
 \bar {\mathcal F}^+_{\bar {\mathcal R}}(\frak r)= 
 \{ f^+\in  \bar {\mathcal F}^+_{\bar {\mathcal R}}: s( f^+) \in   \frak P(\frak r)   \}, \qquad 
 \frak r \in  \widehat{\frak R}^\bullet,
  \end{align*}
 and have obtained partitions
  $$
  \frak P = \bigcup_{\frak r \in \widehat{\frak R} } \frak P( \frak r ),
  \quad
   \bar {\mathcal E}^- =  
   ( \bar {\mathcal E}^- \setminus  \bar {\mathcal F}^-_{\bar {\mathcal R}}) \cup
    \bigcup_{\frak r \in \widehat{\frak R} }  \bar {\mathcal F}^-_{\bar {\mathcal R}}(\frak r),
  \ \
   \bar {\mathcal E}^+ =  
   ( \bar {\mathcal E}^- \setminus  \bar {\mathcal F}^+_{\bar {\mathcal R}}) \cup
    \bigcup_{\frak r \in \widehat{\frak R} }  \bar {\mathcal F}^-_{\bar {\mathcal R}}(\frak r).
  $$
  
For all $\frak r \in \frak R_{\bar {\mathcal R}}^\bullet$ the directed graph 
$G(  \{\frak r  \} \cup \frak P(\frak r)  , \bar {\mathcal F}^-_{\bar {\mathcal R}}(\frak r)  )$ is an outward directed tree, and the directed graph 
$G( \{\frak r  \} \cup \frak P(\frak r)  , \bar {\mathcal F}^+_{\bar {\mathcal R}}(\frak r)  )$ is an inward directed tree. The directed trees 
$G(  \{\frak r  \} \cup \frak P(\frak r)  , \bar {\mathcal F}^-_{\bar {\mathcal R}}(\frak r)  )$
and $G( \{\frak r  \} \cup \frak P(\frak r)  , \bar {\mathcal F}^+_{\bar {\mathcal R}}(\frak r)  )$ are reversals of one another, and the set 
$\frak L( \frak r)$  of their leaves o 
is given by the set of vertices in $\frak P(\frak r)$, that are source vertices of an edge in 
$\bar{\mathcal E}^- \setminus \bar{\mathcal F}^-$, or, equivalently, that are target vertices of an edge in $\bar{\mathcal E}^+ \setminus \bar{\mathcal F}^+$.
To a vertex 
$\frak r \in \frak R_{\bar {\mathcal R}} \setminus \frak R_{\bar {\mathcal R}}^\bullet$
we associate the degenerate tree $G(\{\frak r, \emptyset\})$ with its leaf $ \frak r$.
For $\frak r \in \widehat{\frak R} $, if the source (target) vertex of an edge 
$e^- \in \bar {\mathcal E}^- \setminus  \bar {\mathcal F}^-_{\bar {\mathcal R}} $ 
( $e^+ \in \bar {\mathcal E}^+ \setminus  \bar {\mathcal F}^-_{\bar {\mathcal R}} )$ is in $\frak P(\frak r)$, then this source (target) vertex is necessarily  
in $\frak L(\frak r)$.
 
 From the $\mathcal R$-graph  
  $\mathcal G_{\bar{\mathcal R}}(\frak P, \bar{\mathcal E}^-\cup \bar{\mathcal E}^+)$
 we construct an $\mathcal R$-graph 
 $\mathcal G_{\widehat{\mathcal R}}(\frak P, \widehat{\mathcal E}^-\cup \hat{\mathcal E}^+)$.
 We postulate that there are bijections
 $$
 e^- \to \widehat {e}^- \in \widehat{\mathcal E}^- 
 \ ( e^- \in \bar{\mathcal E}^- \setminus  \bar{\mathcal F}^- ), \quad 
 e^+ \to \widehat {e}^+ \in \widehat{\mathcal E}^+ \ 
 ( e^+\in \bar{\mathcal E}^- \setminus  \bar{\mathcal F}^- ).
 $$
 We specify the source and target mappings of the graph $\mathcal G_{\widehat{\mathcal R}}(\frak P, \widehat{\mathcal E}^-\cup \hat{\mathcal E}^+)$
 by setting
 $$
 t( \widehat {e}^-  ) =  t( e^- ), \qquad   e^-\in   \bar{\mathcal E}^-,
  $$
  $$
 s( \widehat {e}^+  ) =  s( e^+ ), \qquad   e^+\in   \bar{\mathcal E}^-,
  $$
and by assigning for $\frak r \in \widehat {\frak R}$, to an edge $e^-\in   \bar{\mathcal E}^- $, such that  
$s( \widehat {e}^-  )\in \frak L(\frak r)$, the vertex $\frak r$  as source vertex, and to an edge  
  $e^+\in   \bar{\mathcal E}^+ $,  
 such that $t( \widehat {e}^-  )\in \frak L(\frak r)$, the vertex $\frak r$  as target vertex. We set
$$
\widehat{\mathcal R}=
 \{( \widehat {e}^- ,\widehat {e}^+  ) \in \widehat{\mathcal E}^- \times 
  \widehat{\mathcal E}^+ : (  {e}^- , {e}^+  )  \in \bar{\mathcal R} \cap
   \left(( \bar{\mathcal E}^- \setminus  \bar{\mathcal F}^- ) \times  \bar{\mathcal E}^+\setminus  \bar{\mathcal F}^+ ) )\right)  \}.
$$
In the $\mathcal R$-graph $\mathcal G_{\widehat{\mathcal R}}(\frak P, \widehat{\mathcal E}^-\cup \hat{\mathcal E}^+)$
every vertex has at least two incoming edges in $\widehat{\mathcal E}^-$, which means that every vertex in $\widehat{\mathcal E}^+$ has at least two outgoing edges in $\widehat{\mathcal E}^+$.

We introduce additional notation for subshifts. We denote the set of periodic points in 
$A(X)$ by $P(A(X))$. The subshifts $X\subset \Sigma ^\Bbb Z$:
that we consider in this paper are such that $P(A(X))$ is dense in $X$.  We introduce a preorder relation $\gtrsim$ into the set $P(A(X))$ where for $q, r \in  P(A(X))$ means that there exists a point in $A(X)$ that is left asymptotic to the orbit of $q$ and right asymptotic to the orbit of $r$. The equivalence relation on $P(A(X))$ that results from the preorder relation $\gtrsim$ we denote by $\approx$.

The proof of the following theorem is similar to a proof for Markov-Dyck shifts 
 in \cite {KM2}.

\begin{theorem}
Let $\mathcal G_\mathcal R(\frak P, \mathcal E^-, \mathcal E^+)$   be an $\mathcal R$-graph, that satisfies conditions (I) and (II), and let
\begin{align*}
\frak P \setminus{\frak P}_\mathcal R^{(1)} \neq \emptyset.
\end{align*}
Let $X(\mathcal V,\Sigma, \lambda)$ be an $\mathcal S_\mathcal R(\frak P, \mathcal E^-, \mathcal E^+)$-presentation,  let $q, r$ be neutral periodic points of $X(\mathcal V,\Sigma, \lambda)$, and  let $ \frak q, \frak r \in \frak P$,  and
 $I(q), I(r) \in \Bbb Z$ be such that
 $$
p_{( I(q),I(q)+  \pi(q) ]} = \bold 1_\frak q, \quad
p_{( I(r), I(r)+  \pi(r) ]} = \bold 1_\frak r.
$$
 Then  $ q \approx r $ if and only if  $\frak q$ and $\frak r$ are in the same element of the partition $\{\frak P (\frak p): \frak p \in \widehat{\mathcal R} \}$.
 \end{theorem}
\begin{proof}
Assume that  $ p \approx q $. Let $N\in \Bbb N$ and
 \begin{align*}
x^{(q,r)},x^{(r,q)} \in A_N(X(\mathcal V,\Sigma, \lambda)), \tag {3.2}
 \end{align*} 
and let
$$
I_-(q,r), I_+(q,r),\quad I_-(r,q), I_+(r,q) \in \Bbb Z,
$$
$$
I_-(q,r) < I_+(q,r) ,  \quad   I_-(r,q)   <   I_+(r,q)  ,  
 $$
 be such that
 $$
 x^{(q,r)}_{(- \infty, I_-(q,r) ]} =q_{(- \infty, I(q) ]},
 \quad
 x^{(q,r)}_{(I_+(q,r) ,\infty)}=r_{(I(r) ,\infty)},
 $$
  $$
 x^{(r,q)}_{(- \infty, I_-(r,q)]}=r_{(- \infty, I(r)]},
 \quad
 x^{(r, q)}_{(I_+(r,q),\infty)}=q_{( I(q),\infty)}.
 $$
 By (3.2)
  \begin{align*}
 ( x^{(q,r)}_{(- \infty, I_+(q,r) ]}, r_{(I(r)   , I(r) +N\pi (r)   ]},x^{(r,q)}_{( I_-(q,r),\infty)})
  \in X(\mathcal V,\Sigma, \lambda), \tag{3.3}
  \end{align*}
  and
  \begin{align*}
 ( x^{(r,q)}_{(- \infty, I_+(r,q) ]}, q_{(I(q)   , I(q) +N\pi (q)   ]},x^{(q,r)}_{( I_-(r,q),\infty)})
  \in X(\mathcal V,\Sigma, \lambda). \tag {3.4}
   \end{align*}
  As a consequence of (3.2) there are also
  $$
  \frak p{(q,r)}\in \frak P,\quad J_-(q,r), J_+(q,r)\in \Bbb N,
$$
and
$$
e^+_{j_+(q,r) }(q,r)\in \mathcal E^+_\mathcal R, \quad J_+(q,r)\geq j_+(q,r)  > 0,
$$
$$
e^-_{j_-(q,r) }(q,r)\in \mathcal E^-_\mathcal R,\quad 0 < j_-(q,r)  \leq J_-(q,r),
$$
  such that
   \begin{multline*}
  \lambda( x^{(q,r)}_{(I_-(q. r), I_+(q,r) ]} ) =  \tag {3.5}\\
  \bold1_\frak q(\prod_{J_+(q,r)\geq j_+(q,r) >0}e^+_{j_+}(q,r))\bold1_{\frak p^{(q,r)}}(\prod_{0< j_-(q,r) \leq J_-(q,r)}e^-_{j_-(q,r)}(q,r))\bold1_\frak r,
   \end{multline*}
  and there are also
  $$
  \frak p{(r,q)}\in \frak P,\quad J_-(r,q), J_+(r,q)\in \Bbb N,
$$
and
$$
e^+_{j_+(r,q) }(q,r)\in \mathcal E^+_\mathcal R, \quad J_+(r,q)\geq j_+(r,q)  > 0,
$$
$$
e^-_{j_-(r,q) }(q,r)\in \mathcal E^-_\mathcal R,\quad 0 < j_-(r,q)  \leq J_-(r,q),
$$
  such that
    \begin{multline*}
     \lambda(x^{(r.q)}_{(J_-(r.q), J_+(r.q) ]}  )  = \tag {3.6} \\
     \bold1_\frak r(\prod_{J_+(r,q)\geq j_+(r,q) >0}e^+_{j_+}(r,q))\bold1_{\frak p^{(r,q)}}(\prod_{0< j_-(r,q) \leq J_-(r,q)}e^-_{j_-(r,q)}(r,q))\bold1_\frak q.
   \end{multline*}
 By (3.3), and in case that $ J_+(q,r)\geq J_+(r,q)$
    \begin{align*}
   & \bold1_\frak q(\prod_{J_+(q,r)\geq j_+(q,r) >0}e^+_{j_+}(q,r))\bold1_{\frak p^{(q,r)}}(\prod_{0< j_-(q,r) \leq J_-(q,r)}e^-_{j_-(q,r)}(q,r))\bold1_\frak r
   \\
   & \bold1_\frak r(\prod_{J_+(r,q)\geq j_+(r,q) >0}e^+_{j_+}(r,q))\bold1_{\frak p^{(r,q)}}(\prod_{0< j_-(r,q) \leq J_-(r,q)}e^-_{j_-(r,q)}(r,q))\bold1_\frak q =
   \\
  &  \bold1_\frak q(\prod_{J_+(q,r)\geq j_+(q,r) >0}e^+_{j_+}(q,r))\bold1_{\frak p^{(q,r)}}
  \\
 & (\prod_{0< j_-(q,r) \leq J_-(q,r)-J_+(r,q)}e^-_{j_-(q,r)}(q,r))\bold1_{\frak p^{(r, q)}}
 (\prod_{0< j_-(r,q) \leq J_-(r,q)}e^-_{j_-(r,q)}(r,q))\bold1_\frak q \neq 0.
 \end{align*}
  From this it follows  by Condition (I) that $$ (\widehat{e}^+_{j_+}(q,r))_{J_+(q,r)\geq j_+(q,r) >0} $$ is a path in
  $\widehat { \mathcal E}^+(1)$ 
  from $ \frak q$ to $\frak p^{(q,r)} $, and 
 $$((\widehat{e}^-)_{0< j_-(q,r) \leq J_-(q,r)-J_+(r,q)}  ),(\widehat{e}^-)_{0< j_-(r,q) \leq J_-(r,q)}  ))$$
   is a path in 
   $\widehat {\mathcal E}^-(1)$ 
   from $\frak p^{(q,r)} $ to $\frak r$  that passes through $\frak p^{(r,q)} $ 
    and one sees that $\frak p^{(q,r)}, \frak p^{(r,q)}$ and $\frak q$ are in the same element of the partition $\{\frak P _\frak p: \frak p\in \frak P \setminus{\frak P}_\mathcal R^{(1)}  \}$.
  By the same argument for the case that $ J_+(q,r)\leq J_+(q,r)$, and by the symmetric argument 
  that uses (3.6), one sees that in fact $ \frak p^{(q,r)}, \frak p^{(r,q)}$ and $ \frak q$ and $\frak r$ 
   are in the same element of the partition $\{\frak P _\frak p: \frak p\in \frak P \setminus{\frak P}_\mathcal R^{(1)}  \}$.
   
   For the proof of the converse let $p \in  \frak P \setminus{\frak P}_\mathcal R^{(1)}$ and let $\frak q, \frak r \in \frak P _\frak p$.  There is a  path 
   $ (\widehat{e}^+_{j_+(q)})_{J_+(q)\geq j_+(q) >0} $
   in $ \widehat {\mathcal E}^+(1)$ from  $\frak q$ to $\frak p$,
  a path
   $ (\widehat{e}^-_{j_-}(r))_{0<j_-(r) \leq J_-(r) } $
  in $ \widehat {\mathcal E}^-(1)$ from $\frak p$ to $\frak r$,
  a  path 
   $ (\widehat{e}^+_{j_+(r)})_{J_+(r)\geq j_+(r) >0} $
   in $ \widehat {\mathcal E}^+(1)$ from  $\frak r$ to $\frak p$,
  and a path
   $ (\widehat{e}^-_{j_-}(q))_{0<j_-(r) \leq J_-(q) } $
  in $ \widehat {\mathcal E}^-(1)$ from $\frak p$ to $\frak q$.
 Choose a vertex $V(p) \in \mathcal V(\frak p)$ and choose 
 $$
  {e}^+_{j_+(q)} \in \widehat{e}^+_{j_+(q)},\quad J_+(q)\geq j_+(q) >0 
  $$
 and choose a path 
 $b^+(q)$ in the directed graph $G(\mathcal V, \Sigma)$ from the target vertex of $q_{(-\infty,I(q)]}$  to $V(p)$, such that
 $$
 \lambda( b^+(q)  ) = \prod_{J_+(q)\geq j_+(q) >0}  {e}^+_{j_+(q)}, 
 $$
 and choose 
 $$ {e}^+_{j_-(r)} \in \widehat{e}^+_{j_-(r)},\quad 0<j_-(r) \leq J_-(r)  $$
and also a path $b^-(r)$  from $V(p)$  to the source vertex of $q_{(I(r), \infty)}$, such that
 $$
  \lambda( b^-(r)  ) = \prod_{0<j_-(r) \leq J_-(r) }{e}^-_{j_-}(r),
 $$
 choose 
 $$ {e}^+_{j_+(r)} \in \widehat{e}^+_{j_+(r)},\quad J_+(q)\geq j_+(r) >0 $$
 and choose a path 
 $b^+(q)$ in the directed graph $G(\mathcal V, \Sigma)$ from the target vertex of $q_{(-\infty,I(q)]}$  to $V(p)$, such that
 $$
 \lambda( b^+(r)  ) = \prod_{J_+(r)\geq j_+(r) >0}  {e}^+_{j_+(r)},
 $$
 and choose 
 $$ {e}^-_{j_-(q)} \in \widehat{e}^-_{j_-(q)},\quad 0<j_-(q) \leq J_-(q)  $$
and also a path $b^-(q)$  from $V(p)$  to the source vertex of $q_{(I(q), \infty)}$, such that
 $$
  \lambda( b^-(q)  ) = \prod_{0<j_-(q) \leq J_-(q) }{e}^-_{j_-}(q),
 $$
 Then
   \begin{align*}
 &( q_{(- \infty , I(q)]  } , b^+(q) , b^-(r) , r_{(I(r) ,   \infty ) }  ) \in A(X(\mathcal V,\Sigma, \lambda  )),\\
& ( r_{( - \infty , I(r) ]     } , b^+(r) , b^-(q) , q_{ (I(q)    ,  \infty )}  ) \in A(X(\mathcal V,\Sigma, \lambda )).
\qed
 \end{align*}
\renewcommand{\qedsymbol}{}
\end{proof}

We  recall at this point the construction of the associated semigroup. For a property $( A)$ subshift $X \subset \Sigma ^{ \Bbb Z}$ we denote by  $Y(X)$ the set of points in $X$
that are left asymptotic to a  point in $P(A(X))$ and also right-asymptotic to a  point 
in $P(A(X))$ .  Let $y, \tilde{y}\in Y(X),$ let $y$ be left asymptotic to $q \in P(A(X))$ and right 
asymptotic to $r \in P(A(X) ) ,$ and let 
$\tilde {y}$ be left asymptotic to $ \tilde{q} \in P(A(X))$ and right
asymptotic to $ \tilde{r} \in P(A(X))$. Given that 
$X$ has the properties $(a, n, H_{n}), n \in \Bbb N,$ we say
that $y$ and $\tilde {y}$ are equivalent, $y \approx \tilde {y}$, if $q \approx
\tilde{q}$ and $ r \approx \tilde{r}$, and  
if for $n \in \Bbb N$ such that $ q, r, \tilde {q}, \tilde {r} 
\in P(A_n(X))$ and  for $I, J, \tilde
{I}, \tilde {J} \in \Bbb Z,  I < J, \tilde {I}< \tilde
{J},$ such that
$$
y_{(- \infty, I]} = q_{(- \infty, 0]},\quad  y_{(  J,\infty)} = r_{(  0,\infty)},
$$
$$
\tilde{y}_{(- \infty, \tilde{I}]} = \tilde{q}_{(- \infty, 0]},\quad \tilde{ y}_{(  \tilde{J},\infty)} = \tilde{r}_{( 
 0,\infty)},
$$
one has for $h\ge3 H_{n}$ and for 
$$ 
a \in X_{(I  -  h,J + h]}, \quad
\tilde {a} \in X_{(\tilde {I} -  h,\tilde {J} + h]}, 
$$
such that  
$$
a _{(I-  H_{n} ,  J + H_{n} ]} = y_{(I - H_{n} ,  J + H_{n} ]} ,\quad
 \tilde {a} _{(\tilde {I} -  H_{n} ,  \tilde {J} + H_{n} ]} =
\tilde{y}_{(\tilde {I} -  H_{n} ,  \tilde {J} + H_{n} ]},
$$
$$
a _{(I - h ,  I  - h + H_{n})} = \tilde {a} _{(\tilde {I} -
h , \tilde {I}  - h+ H_{n} )},
$$
$$
a _{(J + h- H_{n} ,  J  + h ]} =
 \tilde {a} _{(\tilde {J }+ h - H_{n}  , 
\tilde { J}  + h ]},
$$
and such that
$$
a_{(I - h, I]} \in A_{n}(X)_{(I - h, I]} , \quad
\tilde{a}_{( \tilde{J}- h, \tilde{I}]}\in A_{n}(X)_{( \tilde{J}- h, \tilde{I}]}\ ,
$$
$$
 a_{(J , J+h]} \in A_{n}(X)_{(J , J+h]}  , \quad
  \tilde{a}_{( \tilde{J}  , \tilde{J}+h]}  \in A_{n}(X)_{( \tilde{J}  , \tilde{J}+h]} ,
$$
that $a$ and $\tilde {a}$ have the same context.
To give $[Y(X)]_{\approx}$ the
structure of a semigroup, let $u, v \in Y(X)$, let $u$ be right asymptotic to
$q \in P(A(X))$ and let  $v$ be left asymptotic to
 $r \in P(A(X))$. 
If here 
$q \gtrsim r$, then $[u]_{\approx}[v]_{\approx}$ is set equal to $[y]_{\approx}$
where $y$ is any point in $Y$ such that there are
$n \in \Bbb N, I, J ,\hat {I},  \hat {J} \in \Bbb Z , I < J, \hat 
{I}<   \hat {J}, $ such that $ q, r \in A_{n}(X),$  and such that
$$
u_{(I, \infty)} = q_{(I, \infty)}, \quad
v_{(-\infty, J]} = r_{(-\infty, J]},
$$  
$$
y_{(- \infty,\hat{I} + H_{n}]}= u_{(- \infty,I + H_{n}]} ,\quad
y_{ (\hat{J}  - H_{n}, \infty)}= v_{ (J - H_{n}, \infty)},
$$
and 
$$
y_{(\hat{I}  , \hat{J}  ]}\in A_{n}(X)_{(\hat{I}  , \hat{J}  ]},
$$
provided that such a point $y$ exists. If such a point $y$ does not exist,  
$[u]_{\approx}[v]_{\approx}$ is set equal to zero.
Also, in the case that
one does not have $q \gtrsim r,[u]_{\approx}[v]_{\approx}$ is set equal to zero.

\begin{theorem}
Let $ G_\mathcal R(\frak P, \mathcal E^-\cup \mathcal E^+)$   be an $\mathcal R$-graph, such that
$$
\frak P \setminus{\frak P}_\mathcal R^{(1)} \neq \emptyset,
$$
that satisfies conditions (I) and (II), and let $X(\mathcal V,\Sigma, \lambda)$ be an 
$\mathcal S(G_\mathcal R(\frak P, \mathcal E^+\cup \mathcal E^-))$-presentation. Then
the semigroup
$\mathcal S(  G
_{ \widehat{\mathcal R}}
(\widehat{\frak P},\widehat{ \mathcal E}^+\cup \widehat{\mathcal E}^-) )$
 is associated to $X(\mathcal V,\Sigma, \lambda)$, and  
 $X(G_\mathcal R(\frak P, \mathcal E^+\cup \mathcal E^-))$ has an 
 $\mathcal S(   G_{ \widehat{\mathcal R}}
(\widehat{\frak P},\widehat{ \mathcal E}^+\cup \widehat{\mathcal E}^-)  )$-presentation.
\end{theorem}
\begin{proof}
We choose for $\frak r \in \widehat{\frak R} $ a periodic point 
$y^{\circ } (\frak r )\in X(\frak P,\Sigma, \lambda)$, such that
$$
\lambda( _{[0,\pi(y^{\circ }(\frak r )))} = \bold 1_{\frak r}.
$$
By Lemma (3.1) and by Theorem (3.2)  we can choose  a system of representatives 
$Y^\circ$ of the equivalence relation $\approx$
such that
$
Y^\circ \supset 
\{  y^{\circ } 
(\frak r ): 
 \frak r \in \frak R_{\widehat{\mathcal R}}  \},
$
and such that every point in $Y^\circ$ is left asymptotic to a point in $\{  y^{\circ } 
(\frak r ): 
 \frak r \in \frak R_{\widehat{\mathcal R}}  \}$ and also right asymptotic to a point in
  $\{  y^{\circ } 
(\frak r ): 
 \frak r \in \frak R_{\widehat{\mathcal R}}  \}$.

We set
\begin{align*}
&\varphi( e^- ) = \widehat{[e^-]_{\sim(\mathcal R, -)}},  
\qquad  \qquad\qquad \
 e^- \in\mathcal E^-,
\\
&\varphi( f^- ) = \bold 1_\frak r,\qquad  [f^-]_{\sim(\mathcal R, -)} \in \bar {\mathcal F}^-_{\bar {\mathcal R}}(\frak r),\frak r \in \widehat{\frak R},
\\
&\varphi( \bold 1_\frak p) = \bold 1_\frak r, \qquad \quad \quad\ \  \ \ \ \ \qquad \frak p\in  \frak P_\frak r, \frak r \in \widehat{\frak R},
\\
&\varphi( f^+ ) = \bold 1_\frak r,\qquad  [f^+]_{\sim(\mathcal R, +)} \in \bar {\mathcal F}^+_{\bar {\mathcal R}}(\frak r),\frak r \in \widehat{\frak R},
\\
&\varphi( e^+ ) = \widehat{[e^+]_{\sim(\mathcal R, +)}},  
\qquad  \qquad\qquad \ 
e^+ \in\mathcal E^+.
\end{align*}
and for
$$
g(-) \in \mathcal S^-(G_\mathcal R(\frak P, \mathcal E^-\cup \mathcal E^+) ) , \qquad g(+) \in \mathcal S^+(G_\mathcal R(\frak P, \mathcal E^-\cup \mathcal E^+) )  ,
$$
we set
$$
\varphi( g^- ) = \prod_{1 \leq i (-) \leq \ell(g(-))}\varphi(e_{i(-)} [g(-)]),
\quad
\varphi( g^+ ) = \prod_{1 \leq i (+) \leq \ell(g(+))}\varphi(e_{i(+)} [g(+)] ).
$$
We set
$$
\widehat{\lambda}(\sigma) = \varphi({\lambda}(\sigma)), \qquad \sigma \in \Sigma,
$$
and
\begin{align*}
&J_-(y^\circ) = \max \{ J < 0: y^\circ_{(-\infty, J ]} \in 
\{ y^\circ(\frak r)_{(-\infty, J ]}: \frak r \in \widehat{\frak R} \} \}  ,
\\
&J_-(y^\circ) = \min\thinspace  \{ J > 0: y^\circ_{[J,-\infty)} \in \thinspace 
\{ y^\circ(\frak r)_{[J,-\infty)}: \frak r \in \widehat{\frak R} \} \}, \qquad y^\circ \in Y^\circ.
\end{align*}
An isomorphism $\Psi$ of $\mathcal S(X(  \lambda)$ onto $\mathcal S(\widehat{G}()$ is given by
$$
\Psi([(y^\circ)]_{\approx}) = \prod_{J_-(y^\circ)< j <J_+(y^\circ)}\widehat{\lambda}
(y^\circ_j), \qquad y^\circ \in Y^\circ. \qed
$$
\renewcommand{\qedsymbol}{}
ore precisely, with $I, J \in \Bbb Z, I < J,$ such that 
\end{proof}

\section{Examples}

Let there be given an $\mathcal R$-graph 
$$
G_{\widehat{\mathcal R}} = G_{\widehat{R}}
 ( \{\widehat{ \frak p}\} , \widehat{\mathcal E}^- \cup 
\widehat{ \mathcal E}^+),
$$
such that 
$$
\card(\widehat{\mathcal E}^-) > 1, 
$$
$$
[\widehat{e}^-] _{\sim \widehat{\mathcal R}, - }  = \{ \widehat{e}^- \},  \qquad 
\widehat{e}^- \in \widehat{\mathcal E}^-,
$$
$$
[\widehat{e}^+] _{\sim \widehat{\mathcal R}, -+}  = \{ \widehat{e}^- \},  \qquad 
\widehat{e}^- \in \widehat{\mathcal E}^+.
$$
Let there also be given a subshift $X$ with Property (A) and associated semigroup  $\mathcal S(G_{\widehat{R}})$.
For $\widehat{e}^-\in\widehat{\mathcal E}^-$($\widehat{e}^-\in \widehat{\mathcal E}^-$) we denote by $\Lambda (\widehat{e}^-)$($\Lambda (\widehat{e}^+)$) the minimal length of a periodic point in $X$ with multiplier $\widehat{e}^-$($\widehat{e}^+$), and we denote by 
$P^{(\widehat{e}^-)}_l$   ($P^{(\widehat{e}^+)}_l$) the set of periodic points of $X$ of length $l \geq  \Lambda (\widehat{e}^-) $
($l \geq  \Lambda (\widehat{e}^+)$), with multiplier $\widehat{e}^-$($\widehat{e}^+$). We set 
$$
P^{(-)} =  \bigcup_{l \geq  \Lambda (\widehat{e}^-) }  P^{(\widehat{e}^-)}_l ,   \qquad
P^{(+)} =\bigcup_{l \geq  \Lambda (\widehat{e}^+) }  P^{(\widehat{e}^+)}_l .
$$
\subsection{One-vertex $\mathcal R$-graphs}
\begin{theorem}
For $\mathcal R$-graphs
$$
G_{\mathcal R} = G_{\mathcal R}
 ( \{{ \frak p}\} , {\mathcal E}^- \cup 
{ \mathcal E}^+),
$$
such that 
$$
{ \frak p}  \notin \frak P^{(1)}_{{{\mathcal R}}}, \qquad
\mathcal E^-_\mathcal R( \frak p , \frak p ) =\mathcal E^+_\mathcal R( \frak p , \frak p ) = \emptyset,
$$
the topological conjugacy class 
of the $\mathcal R$-graph shift of $G_{\mathcal R}$ 
determines the isomorphism class of the $\mathcal R$-graph  
$G_{\mathcal R}$.
\end{theorem}
\begin{proof}
We  define a relation 
$
\mathcal Q\subset P_{1}^{(-)}\times P_{1}^{(+)}
$
that contains the pairs 
$
( p(-) ,p(+)   ) \in P_{1}^{(-)}\times P_{1}^{(+)},
$
such that the subshift $X(G_{\mathcal R})$ has a point, that is left asymptotic to $ p(-) $ and right asymptotic to $ p(+) $.

To obtain  mappings
$$
\widehat{\nu}_{-}: \widehat{\mathcal E}^- \to \Bbb N, \  \ \widehat{\nu}_{+}: 
\widehat{\mathcal E}^+ \to \Bbb N, 
$$
such that one has for 
$$
p(-) \in P^{(\widehat{e}^-)}, p(+) \in P^{(\widehat{e}^+)}, \qquad
( \widehat{e}^- ,  \widehat{e}^+ ) \in \widehat{\mathcal R},
$$
that 
\begin{align*}
&\card(\{\widetilde{p}(+) \in P_1^{(\widehat{e}^-)}(X): (p(-) ,\widetilde{p}(+) ) 
\in  \mathcal Q   \}) =
\widehat{\nu}_{-}(\widehat{e}^-), \tag{4.1}
\\
&\card(\{\widetilde{p}(-) \in P_1^{(\widehat{e}^-)}(X): (\widetilde{p}(-) ,p(+) )
 \in  \mathcal Q   \}) =
\widehat{\nu}_{+}(\widehat{e}^+),
\end{align*}
set
$$
\widehat{\nu}_{-}(\widehat{e}^-)  = \card(\widehat{e}^-), \qquad \widehat{e}^-\in 
\widehat{\mathcal E}^-,
$$
$$
\widehat{\nu}_{+}(\widehat{e}^+)  = \card(\widehat{e}^+). \qquad \widehat{e}^+ \in 
\widehat{\mathcal E}^+.
$$
The existence of mappings
$$
\widehat{\nu}_{-}: \widehat{\mathcal E}^- \to \Bbb N, \  \ \widehat{\nu}_{+}: 
\widehat{\mathcal E}^+ \to \Bbb N, 
$$
such that (4.1) holds is an invariant of topological conjugacy. Given the mappings 
$$
\widehat{\nu}_{-}: \widehat{\mathcal E}^- \to \Bbb N, \  \ \widehat{\nu}_{+}: 
\widehat{\mathcal E}^+ \to \Bbb N,
$$
such that (4.1) holds, together with
the  $\mathcal R$-graph
$$
G_{\widehat{R}} = G_{\widehat{R}}
 ( \{ \frak p\} , [\mathcal E^-]_{\sim(\mathcal R, -)} \cup 
[\mathcal E^+]_{\sim(\mathcal R, +)}),
$$
that is derived from  $G_{\mathcal R}$,
one can reconstruct  the  $\mathcal R$-graph $G_{\mathcal R}$ from the relation
 $\widehat{R}$. 
\end{proof}

\subsection{A class of Examples of $\mathcal R$-graph shifts, to which the $\mathcal R$-graph semigroups of one-vertex $ \mathcal R$-graphs are associated.}

Let $K \in \Bbb N.$ We set
\begin{align*}
&\widehat{\mathcal E}^-_l(X) = \{\widehat{e}^- \in\widehat{ \mathcal E}^-: 
\Lambda(\widehat{e}^-) =  l + 1  \},
\\
&\widehat{\mathcal E}^+_l(X) = \{\widehat{e}^+ \in\widehat{ \mathcal E}^+: 
\Lambda(\widehat{e}^+) =  l + 1  \}, \quad 0 < l \leq K,
\end{align*}
and
$$
\widehat{\mathcal R}_l(X)= \widehat{\mathcal R}
\cap( \widehat{\mathcal E}^-_l(X) \times\widehat{\mathcal E}^+_l(X) ), \quad 0 < l \leq K.
$$
We also define relations $$
\mathcal Q_l(X) \subset P_{l+1}^{(-)}(X)\times P_{l+1}^{(+)}(X), \quad  0< l \leq K,
$$
$\mathcal Q_l(X)$ containing the pairs 
$$
( p(-) ,p(+)   ) \in P_{l+1}^{(-)}(X)\times P_{l+1}^{(+)}(X),
$$
such that the subshift $X$ has a point, that is left asymptotic to $ p(-) $ and right asymptotic to $ p(+) $.

We denote by $\mathcal H(G_{\widehat{R}} )$ the family of subshifts $X$ with Property (A) and associated semigroup 
$\mathcal S(G_{\widehat{R}})$, that satisfy the following conditions $(c1 - 4)$,
that are designed to be invariant under topological conjugacy:

\noindent
\begin{align*}
\widehat{\mathcal E}^-  =\prod_{0< l \leq K}\widehat{\mathcal E}^-_l(X), \qquad
 \widehat{\mathcal E}^+ = \prod_{0< l \leq K}\widehat{\mathcal E}^+_l(X).
\tag {c1} 
\end{align*}

\begin{align*}
\widehat{\mathcal R}  =\bigcup_{0< l \leq K}\widehat{\mathcal R}_l(X).
\tag {c2} 
\end{align*}

\noindent
(c3) There are $\Xi_{K,l}\in \Bbb N, 0< l \leq K, $ such that
\begin{align*}
\Xi_l = \card(P^{(e^-)}_{l+3})  =
\card(P^{(e^+)}_{l+3})   , \quad \
\widehat{e}^-\in\widehat{\mathcal E}^-, \widehat{e}^-\in\widehat{\mathcal E}^-,\ 0< l \leq K.
\end{align*}

\noindent
(c3) There are mappings
$$
\widehat{\nu}_{-, l}: \widehat{\mathcal E}^-_l \to \Bbb N, \  \ \widehat{\nu}_{+, l}: 
\widehat{\mathcal E}^+_{l} \to \Bbb N,  \qquad 0< l \leq K.
$$
such that it holds for 
$$
p(-) \in P_l^{(\widehat{e}^-)}(X), p(+) \in P_l^{(\widehat{e}^+)}(X), \qquad
( \widehat{e}^- ,  \widehat{e}^+ ) \in \mathcal R_l(X)  (X),
$$
that 
\begin{align*}
&\card(\{\widetilde{p}(+) \in P_l^{(\widehat{e}^-)}(X): (p(-) ,\widetilde{p}(+) ) \in  \mathcal Q_l(X)   \}) =
\widehat{\nu}_{-, l}(\widehat{e}^-),
\\
&\card(\{\widetilde{p}(-) \in P_l^{(\widehat{e}^-)}(X): (\widetilde{p}(-) ,p(+) ) \in  \mathcal Q_l(X)   \}) =
\widehat{\nu}_{+, l}(\widehat{e}^+). 
\end{align*}

Let $K \in \Bbb N$, and let there be given an $\mathcal R$-graph
$$
G_{\mathcal R}= G_{\mathcal R}(\{\frak p_{l}:  0 \leq l \leq K  \}, 
\mathcal E^- \cup \mathcal E^-),
$$
such that
\begin{align*}
\mathcal E^-(\frak p_{l} , \frak p_{0} ) \neq \emptyset, \quad
\mathcal E^-( \frak p_{l-1} , \frak p_{l} ) \neq \emptyset,  \qquad
 0 \leq l \leq K \tag {4.2},
\end{align*}
\begin{align*}
\mathcal E^-  = \bigcup_{0 \leq l \leq K} (\mathcal E^-(\frak p_{l} , \frak p_{0} ) \cup
 \mathcal E^-( \frak p_{l-1} , \frak p_{l} ) ),              \tag{4.3}
\end{align*}

\begin{align*}
\frak p_{l} \in \frak P^{(1)}_\mathcal R, \qquad  0 \leq l \leq K, \tag{4.4}
\end{align*}
In the case that  $K = 1$, also assume, that 
\begin{align*}
\frak p_{0} \in  \frak P^{(1)} \setminus \frak P^{(1)}_\mathcal R.
\end{align*}

Let $G_{\widehat{R}} = G_{\widehat{R}}
 ( \{\widehat{ \frak p}, \widehat{\mathcal E}^- \cup 
\widehat{ \mathcal E}^+ \}    )$ 
be the $\widehat{R}$-graph, that is derived from $G_\mathcal R $.
We set
$$
\widehat{R}_l = \widehat{R} \cap 
( [ \mathcal E^-(\frak p_{l} , \frak p_{0}  )]_{\sim(\mathcal R, +)}     \times 
  [ \mathcal E^+(\frak p_{l} , \frak p_{0}  )]_{\sim(\mathcal R, -)}    ), \quad   0 \leq l \leq K ,
$$
and we introduce a condition (D) on the $\mathcal R$-graph $G_{\mathcal R}$, that consists of two parts, $(Da)$ and $(Db)$:

\noindent
\begin{align*}
gcd( \card(\mathcal E^-(\frak p_{l-1} , \frak p_{l} ) \card(\mathcal E^+(\frak p_{l-1} ,
 \frak p_{l} )), \sum_{(\widehat{e}^-,\widehat{e}^+) \in 
  \widehat{R}_l}\card(\widehat{e}^-)\card(\widehat{e}^+))=1, \tag{Da}
 \end{align*}
 \begin{multline*}
  gcd\{\card(\widehat{e}^-): \widehat{e}^- \in   
  [ \mathcal E^-(\frak p_{l} , \frak p_{0}  )]_{\sim(\mathcal R, -)} \} =  \tag{Db}
  \\ 
 gcd\{\card(\widehat{e}^-): \widehat{e}^- \in 
    [ \mathcal E^+(\frak p_{l} , \frak p_{0}  )]_{\sim(\mathcal R, +)} \} =1, \quad
     0 \leq l \leq K.
\end{multline*}

\begin{theorem}
Let $K \in \Bbb N$, and let 
$$
G_{\mathcal R}= G_{\mathcal R}(\{\frak p_{l}:  0 \leq l \leq K  \}, 
\mathcal E^- \cup \mathcal E^-),
$$
be  an $\mathcal R$-graph,
that satisfies (4.2), (4.3) and (4.4) 
and such that, in the case that  $K = 1$, also 
$
\frak p_{0} \in  \frak P^{(1)} \setminus \frak P^{(1)}_\mathcal R.
$

Condition $(D)$ is an invariant of topological conjugacy of the $\mathcal R$-graph shift of 
$G_{\mathcal R}$. If the $\mathcal R$-graph shift of 
$G_{\mathcal R}$ satisfies Condition (D), then its topological conjugacy class determines the isomorphism class of the $\mathcal R$-graph  
$G_{\mathcal R}$.
\end{theorem}
 
\begin{proof}
Let $G_{\widehat{R}} = G_{\widehat{R}}
 ( \{\widehat{ \frak p}\}, \widehat{\mathcal E}^- \cup 
\widehat{ \mathcal E}^+ )$ 
be the $\widehat{R}$-graph, that is derived from
$G_{\mathcal R}.$ The $\mathcal R$-graph shift of $G_{\mathcal R}$ belongs to the family 
$\mathcal H_K(G_{\widehat{R}})  $: By construction  $(c1)$ and $(c2)$ are satisfied. One has, that
\begin{align*}
&\widehat{\mathcal E}^-_{l}(X(G_{\mathcal R})) =  
 [ \mathcal E^-(\frak p_{l} , \frak p_{0}  )]_{\sim(\mathcal R, +)} ,
 \ \ \
\widehat{\mathcal E}^+_{l}(X(G_{\mathcal R})) = 
 [ \mathcal E^-(\frak p_{l} , \frak p_{0}  )]_{\sim(\mathcal R, -)}, 
\\
&\widehat{\mathcal R}_{l}(X(G_{\mathcal R})) = \widehat{\mathcal R}
 \cap ( \widehat{\mathcal E}^-_{l}(X(G_{\mathcal R}))   \times 
  \widehat{\mathcal E}^+_{l}(X(G_{\mathcal R}))  ), \qquad \qquad \ 0 \leq l \leq K.
\end{align*}
Condition (c3) is also satisfied: One has that
\begin{multline*}
\Xi_{K, l}(X(G_{\mathcal R})) = \tag{4.5}
\\
\sum_{0< m \leq l}( \card (\mathcal E^-(\frak p_{m-1}  ,  \frak p_{m} ))
\card(\mathcal E^+(\frak p_{m-1}  ,  \frak p_{m} )) +
\sum_{(  \widehat{e}^-,  \widehat{e}^+)\in \widehat{\mathcal R} _m}
\card(\widehat{e}^-) 
\card (\widehat{e}^+)),
\\
0 < l \leq K.
\end{multline*}
The Condition (c4) is also satisfied. One has, that
\begin{align*}
&\nu_{-, l}^{X(G_{\mathcal R})}(\widehat{e}^-) = \card (\widehat{e}^-) \prod_{0 < m \leq l}\card(\mathcal E^-(\frak p_{m-1}  ,  \frak p_{m} )), \quad \widehat{e}^- \in \mathcal E^-_l,
\\
&\nu_{-, l}^{X(G_{\mathcal R})}(\widehat{e}^+) = \card (\widehat{e}^+) \prod_{0 < m \leq l}\card(\mathcal E^+(\frak p_{m-1}  ,  \frak p_{m} )), \quad \widehat{e}^+ \in \mathcal E^+_l, \qquad 0 < l \leq K.
\end{align*}
We set
$$
\rho_{-, l}^{X(G_{\mathcal R})} = gcd \{ \nu_{-, l}^{X(G_{\mathcal R})}(\widehat{e}^- ): 
\widehat{e}^- \in \mathcal E^-_l (X(G_{\mathcal R})) \},
$$
$$
\rho_{+, l}^{X(G_{\mathcal R})}= gcd \{ \nu_{+, l}^{X(G_{\mathcal R})}(\widehat{e}^+ ): 
\widehat{e}^+ \in \mathcal E^+_l (X(G_{\mathcal R})) \}.
$$

The proof of the theorem is by induction in $K$ steps. For the $l$-th  step of the induction we split Condition $(D)$ into conditions $(Dal)$ and $(Dbl)$:
\begin{multline*}
(Dal) \  \ \ \ gcd( \card(\mathcal E^-(\frak p_{l-1} , \frak p_{l} ) \card(\mathcal E^+(\frak p_{l-1} ,
 \frak p_{l} )), \sum_{(\widehat{e}^-,\widehat{e}^+) \in 
  \widehat{R}_l}\card(\widehat{e}^-)\card(\widehat{e}^+))=1, 
\\
(Dbl)   \quad gcd\{\widehat{e}^-: \widehat{e}^- \in   
  [ \mathcal E^-(\frak p_{l} , \frak p_{0}  )]_{\sim(\mathcal R, -)}        \} =   
 gcd\{\widehat{e}^+: \widehat{e}^+ \in 
    [ \mathcal E^-(\frak p_{l} , \frak p_{0}  )]_{\sim(\mathcal R, -)} \} =1, 
    \\
    0 < l \leq K.
\end{multline*}

At the $l$-th step of the induction the cardinalities of the sets
\begin{align*}
\mathcal E^-(\frak p_{m-1} , \frak p_{m} ),   \ 
\mathcal E^+(\frak p_{m-1} , \frak p_{m} ),
\quad
\mathcal E^-(\frak p_m , \frak p_{0}  ),  
 \ \mathcal E^+(\frak p_{m} , \frak p_{0}  ), \qquad  0 < m < l, 
\end{align*}
are known.
In the $l$-th step of the induction it is shown, that the validity of $(Dal)$ is determined by the topological conjugacy class of $X(G_{\mathcal R})$. Under the assumption, that 
$(Dal)$ holds, it is then also shown in the $l$-th step of the induction, that  the validity of 
$(Dbl)$ is determined by the topological conjugacy class of $X(G_{\mathcal R})$. Under the assumption, that $(Dal)$ and $(Dbl)$ hold, it is then shown, in the $l$-th step of the induction, how the cardinality of the edge sets
\begin{align*}
\mathcal E^-(\frak p_{l-1} , \frak p_{l} ),   \ 
\mathcal E^+(\frak p_{l-1} , \frak p_{} ),
\quad
\mathcal E^-(\frak p_l , \frak p_{0}  ),  
 \ \mathcal E^+(\frak p_{l} , \frak p_{0}  ), 
\end{align*}
can be obtained from topological conjugacy invariants of $X(G_{\mathcal R})$.

We describe for $l \in [1, K]$ the $l$-th step of the induction. One has, that
\begin{align*}
&\card(\mathcal Q_l(X(G_{\mathcal R}))) =  \tag{4.6}
\\
&( \prod_{0 < m \leq l}
 (\card(\mathcal E^-(\frak p_{m-1},\frak p_m)\card(\mathcal E^+(\frak p_{m-1},\frak p_m))
(\sum_{(\widehat{e}^-,  \widehat{e}^+)\in \widehat{\mathcal R} _l}
\card( \widehat{e}^-)\card( \widehat{e}^+))>
\\
&
\rho_{-,l}X(G_{\mathcal R}))\rho_{-,l}X(G_{\mathcal R}))\card( \widehat{\mathcal R} _l) >
\rho_{-,l}X(G_{\mathcal R}))\rho_{-,l}X(G_{\mathcal R}).
\end{align*}
Set $\Xi_0(X(G_{\mathcal R})  = 0$, and set
\begin{align*}
&\alpha_l^{X(G_{\mathcal R})} = \Xi_l (X(G_{\mathcal R}))) -  \Xi_{l-1}(X(G_{\mathcal R}))),
\\
\beta_l^{X(G_{\mathcal R})}=& \frac{\card(\mathcal Q_l(X(G_{\mathcal R})))}
{ \prod_{0 < m \leq l}
 (\card(\mathcal E^-(\frak p_{m-1},\frak p_m)\card(\mathcal E^+(\frak p_{m-1},\frak p_m))}, 
\quad  0 < l \leq K.
\end{align*}
From (4.5) one has, that
\begin{align*}
\alpha_l^{X(G_{\mathcal R})} = \card(\mathcal E^-(\frak p_{l-1},\frak p_l)
\card(\mathcal E^+(\frak p_{l-1},\frak p_l)) +
\sum_{(\widehat{e}^-,  \widehat{e}^+)\in \widehat{\mathcal R} _l}\card( \widehat{e}^-)\card( \widehat{e}^+),\tag{4.7}
\end{align*}
and from (a) one has, that
\begin{align*}
\beta_l^{X(G_{\mathcal R})} =\card(\mathcal E^-(\frak p_{l-1},\frak p_l)
\card(\mathcal E^+(\frak p_{l-1},\frak p_l)) 
\sum_{(\widehat{e}^-,  \widehat{e}^+)\in \widehat{\mathcal R} _l}\card( \widehat{e}^-)\card( \widehat{e}^+). \tag{4.8}
\end{align*}
Set
$$
\Delta_l^{X(G_{\mathcal R})} = (\alpha_l^{X(G_{\mathcal R})})^2 - 
4\beta_l^{X(G_{\mathcal R})}.
$$
It follows from (4.7) and (4.8), that
\begin{multline*}
\card(\mathcal E^-(\frak p_{l-1},\frak p_l)
\card(\mathcal E^+(\frak p_{l-1},\frak p_l)) \in 
\\
\{\tfrac{1}{2}(
 \alpha_l^{X(G_{\mathcal R})} -\sqrt{\Delta_l^{X(G_{\mathcal R})}} )
 , \tfrac{1}{2}(
  \alpha_l^{X(G_{\mathcal R})} +\sqrt{\Delta_l^{X(G_{\mathcal R})}} ) \}.
\end{multline*}
At this stage the cardinalities of the sets 
$$
\mathcal E^-(\frak p_{m-1},\frak p_m),\mathcal E^+(\frak p_{m-1},\frak p_m),\qquad 0 < m <l,
$$
are known, and are also known to be invariants of topological conjugacy. This means that   $\beta_l^{X(G_{\mathcal R})}$ is also an invariant of topological conjugacy, as is the set of roots
$$
\{\tfrac{1}{2}(
 \alpha_l^{X(G_{\mathcal R})} -\sqrt{\Delta_l^{X(G_{\mathcal R})}} )
 , \tfrac{1}{2}(
  \alpha_l^{X(G_{\mathcal R})} +\sqrt{\Delta_l^{X(G_{\mathcal R})}} ) \}.
$$
It follows, that the validity of Condition $(Dal)$ is an invariant of topological conjugacy.

Under the assumption, that $(Dal)$ holds, one has, that either 
$ \alpha_l^{X(G_{\mathcal R})} -\sqrt{\Delta_l^{X(G_{\mathcal R})}}$
or 
$\alpha_l^{X(G_{\mathcal R})} +\sqrt{\Delta_l^{X(G_{\mathcal R})}}$
is a divisor of 
$\rho_{-,l}^{X(G_{\mathcal R})}
\rho_{+,l}X(G_{\mathcal R}))$,
since by (4.6) 
$$
\frac{\rho_{-,l}^{X(G_{\mathcal R})}
\rho_{+,l}^{X(G_{\mathcal R})}}
{\card(\mathcal Q_l^{X(G_{\mathcal R})})} <1.
$$
Therefore one has, under the assumption, that $(Dal)$ holds, that
\begin{multline*}
\card(\mathcal E^-(\frak p_{l-1},\frak p_l)
\card(\mathcal E^+(\frak p_{l-1},\frak p_l)) = 
\\
\begin{cases}
 \alpha^{X(G_{\mathcal R})}_{l}-\sqrt{\Delta^{X(G_{\mathcal R})}_l}, 
 &\text{if  
 $  \alpha^{X(G_{\mathcal R})}_{l}-\sqrt{\Delta^{X(G_{\mathcal R})}_l} 
 \mid\rho_{-,l}^{X(G_{\mathcal R})}
\rho_{+,l}^{X(G_{\mathcal R})}$}, 
\\
  \alpha^{X(G_{\mathcal R})}_{l}-\sqrt{\Delta^{X(G_{\mathcal R})}_l}, 
 &\text {if  $ \alpha^{X(G_{\mathcal R})}_{l}-\sqrt{\Delta^{X(G_{\mathcal R})}_l} 
 \mid\rho_{-,l}^{X(G_{\mathcal R})}
\rho_{+,l}^{X(G_{\mathcal R})}$}.
\end{cases} \tag{4.9}
\end{multline*}
Condition $(Dbl)$ is equivalent to
$$
\rho_{-,l}^{X(G_{\mathcal R})}
\rho_{+,l}^{X(G_{\mathcal R})} = \prod_{0< m \leq l}[\card(\mathcal E^-(\frak p_{l-1},\frak p_l)
\card(\mathcal E^+(\frak p_{l-1},\frak p_l))].
$$ 
From this and from (4.9) it is seen, that under the assumption, that Condition $(Dal)$ holds, the validity of Condition $(Dbl)$ is an invariant under topological conjugacy. From Condition $(Dal)$ together with Condition $(Dbl)$ it follows that
$$
\card (\mathcal E^-(\frak p_{l-1},\frak p_{l}) =  \frac{\rho_{-,l}^{X(G_{\mathcal R})}}
{\prod_{0 < m <l} \card(\mathcal E^-(\frak p_{m-1},\frak p_{m})}, \ 
$$
$$
\card (\mathcal E^+(\frak p_{l-1},\frak p_{l}) =  \frac{\rho_{+,l}^{X(G_{\mathcal R})}}
{\prod_{0 < m <l} \card(\mathcal E^+(\frak p_{m-1},\frak p_{m})},
$$ 
and
$$
\card (\widehat{e}^-) =  \frac{\nu_{-,l}^{X(G_{\mathcal R})}}
{\prod_{0 < m <l} \card(\mathcal E^-(\frak p_{m-1},\frak p_{m})},  \qquad  \widehat{e}^-\in
\widehat{\mathcal E}^-, 
$$
$$
\card (\widehat{e}^+)  =  \frac{\nu_{+,l}^{X(G_{\mathcal R})}}
{\prod_{0 < m <l} \card(\mathcal E^+(\frak p_{m-1},\frak p_{m})}, \qquad  \widehat{e}^+\in
\widehat{\mathcal E}^+, 
$$ 
and the relation $\mathcal R_l$ can be reconstructed from
$$
( \card(\widehat{e}^+)  )_{\widehat{e}^+ \in \widehat{\mathcal E}_l^-},   \quad 
( \card(\widehat{e}^-)  )_{\widehat{e}^- \in \widehat{\mathcal E}_l^-},
$$
and from the relation $\widehat{\mathcal R}_l$.
\end{proof}


\medskip

\par\noindent Toshihiro Hamachi
\par\noindent  Faculty of Mathematics
\par\noindent Kyushu University
\par\noindent 744 Motooka, 
 Nishi-ku,
Fukuoka 819-0395 
\par\noindent Japan
\par\noindent t.hamachi.796@m.kyushu-u.ac.jp
 
\bigskip

\par\noindent Wolfgang Krieger
\par\noindent Institute for Applied Mathematics 
\par\noindent  University of Heidelberg
\par\noindent Im Neuenheimer Feld 205, 
 69120 Heidelberg
 \par\noindent Germany
\par\noindent krieger@math.uni-heidelberg.de

 \end{document}